\documentclass{article}

\usepackage[
	a4paper,
	margin=25mm
]{geometry}
\usepackage{setspace}
\usepackage{sectsty} 
\usepackage{titlesec}

\usepackage{amsmath, amssymb, amsthm}
\usepackage{mathtools}

\usepackage[
	setpagesize=false,
	colorlinks=true,
	linkcolor=blue,
	pdfencoding=auto,
	psdextra,
]{hyperref}
\usepackage{cleveref}

\usepackage{etoolbox} 
\usepackage[shortlabels]{enumitem}
\usepackage{xparse} 

\usepackage{graphicx}
\usepackage[dvipsnames]{xcolor}
\usepackage{tikz}

\usepackage{todonotes}
\usepackage[
    deletedmarkup=sout,
    commentmarkup=uwave,
    authormarkuptext=name
]{changes}

\usepackage{comment}

\setlist[enumerate,1]{label=(\arabic*), ref=(\arabic*)}
\setlist[enumerate,3]{label=(\roman*), ref=(\roman*)}

\allsectionsfont{\boldmath} 

\titlelabel{\thetitle.\quad}

\theoremstyle{plain}
\newtheorem{theorem}{Theorem}[section]
\newtheorem{lemma}[theorem]{Lemma}

\newtheorem{proposition}[theorem]{Proposition}
\newtheorem{observation}[theorem]{Observation}
\newtheorem{conjecture}[theorem]{Conjecture}
\newtheorem{problem}[theorem]{Problem}

\newtheorem{claim}{Claim}[theorem]
\newtheorem*{claim*}{Claim}

\makeatletter
\newenvironment{claimproof}[1][Proof]{\par
	\pushQED{\qed}%
	
	\normalfont \topsep6\p@\@plus6\p@\relax
	\trivlist
	\item[\hskip\labelsep
	\textit{#1}\@addpunct{.}~]\ignorespaces
}{%
	\popQED\endtrivlist\@endpefalse
}
\makeatother

\newlist{Cases}{enumerate}{3}
\setlist[Cases]{parsep=0pt plus 1pt}
\setlist[Cases,1]{wide=0pt, listparindent=\parindent,
    label = \textbf{Case~\arabic*:}, ref = \arabic*}
\setlist[Cases,2]{wide=\parindent, listparindent=\parindent,
    label = \textbf{Case~\arabic{Casesi}-\arabic{Casesii}:}}

\crefname{Casesi}{case}{cases}

\newcounter{case}
\AtBeginEnvironment{proof}{\setcounter{case}{0}}

\crefname{case}{case}{cases}

\theoremstyle{definition}
\newtheorem{definition}[theorem]{Definition}

\newcommand{\calB}{\mathcal{B}}
\newcommand{\calC}{\mathcal{C}}
\newcommand{\calD}{\mathcal{D}}
\newcommand{\calF}{\mathcal{F}}

\newcommand{\calP}{\mathcal{P}}

\newcommand{\calT}{\mathcal{T}}

\newcommand{\calY}{\mathcal{Y}}
\newcommand{\calZ}{\mathcal{Z}}

\newcommand{\ve}{\varepsilon}

\newcommand{\hcf}{\mathrm{hcf}}

\makeatletter
\NewDocumentCommand{\xsideset}{mmme{_^}}{%
  \mathop{%
    \settowidth{\dimen0}{$\m@th\displaystyle#3$}%
    \dimen0=.5\dimen0
    \settowidth{\dimen2}{$%
      \m@th\displaystyle#3%
      \IfValueT{#4}{_{#4}}%
      \IfValueT{#5}{^{#5}}%
    $}%
    \dimen2=.5\dimen2
    \advance\dimen2 -\dimen0
    \sbox6{\scriptspace\z@$\displaystyle{\vphantom{#3}}#1$}
    \sbox8{\scriptspace\z@$\displaystyle{\vphantom{#3}}#2$}
    \ifdim\wd6>\dimen2 \kern\dimexpr\wd6-\dimen2\relax\fi
    {%
     \mathop{\llap{\copy6}{\displaystyle#3}\rlap{\copy8}}\limits
     \IfValueT{#4}{_{#4}}%
     \IfValueT{#5}{^{#5}}%
    }%
    \ifdim\wd8>\dimen2 \kern\dimexpr\wd8-\dimen2\relax\fi
  }%
}
\makeatother

\newcommand{\defeq}{\coloneqq}

\let\originalleft\left
\let\originalright\right
\renewcommand{\left}{\mathopen{}\mathclose\bgroup\originalleft}
\renewcommand{\right}{\aftergroup\egroup\originalright}

\makeatletter
\newcases{lrcases*}
  {\quad}
  {$\m@th{##}$\hfil}
  {{##}\hfil}
  {\lbrace}
  {\rbrace}
\makeatother

\usetikzlibrary{matrix,arrows,decorations.pathmorphing}
\usetikzlibrary{positioning}
\usetikzlibrary{graphs} 
\usetikzlibrary{graphs.standard}

\title{On Perfect Subdivision Tilings}

\author{%
    Hyunwoo Lee%
        \thanks{Department of Mathematical Sciences, KAIST, South Korea and Extremal Combinatorics and Probability Group
(ECOPRO), Institute for Basic Science (IBS).
        E-mail: {\ttfamily hyunwoo.lee@kaist.ac.kr.}
        Hyunwoo Lee was supported by the Institute for Basic Science (IBS-R029-C4).} 
}

\begin{document}
\maketitle

\begin{abstract}
    For a given graph $H$, we say that a graph $G$ has a perfect $H$-subdivision tiling if $G$ contains a collection of vertex-disjoint subdivisions of $H$ covering all vertices of $G.$ Let $\delta_{\mathrm{sub}}(n, H)$ be the smallest integer $k$ such that any $n$-vertex graph $G$ with minimum degree at least $k$ has a perfect $H$-subdivision tiling.
    For every graph $H$, we asymptotically determined the value of $\delta_{\mathrm{sub}}(n, H)$. More precisely, for every graph $H$ with at least one edge, there is an integer $\hcf_{\xi}(H)$ and a constant $1 < \xi^*(H)\leq 2$ that can be explicitly determined by structural properties of $H$ such that $\delta_{\mathrm{sub}}(n, H) = \left(1 - \frac{1}{\xi^*(H)} + o(1) \right)n$ holds for all $n$ and $H$ unless $\hcf_{\xi}(H) = 2$ and $n$ is odd. When $\hcf_{\xi}(H) = 2$ and $n$ is odd, then we show that $\delta_{\mathrm{sub}}(n, H) = \left(\frac{1}{2} + o(1) \right)n$.
\end{abstract}

\section{Introduction}\label{sec:intro}

Embedding a large sparse subgraph into a dense graph is one of the most central problems in extremal graph theory. It is well-known that any graph $G$ with a minimum degree of at least $\lceil \frac{v(G)}{2} \rceil$ has a Hamiltonian cycle, and hence also a perfect matching if the number of vertices $v(G)$ of $G$ is even.
A \emph{perfect matching} of a graph $G$ is a vertex-disjoint collection of edges whose union covers all vertices of $G$. A natural generalization of a perfect matching is a perfect $H$-tiling for a general graph $H$.
We say $G$ has a \emph{perfect $H$-tiling} if $G$ contains a collection of vertex-disjoint copies of $H$ whose union covers all vertices of $G$. We note that a perfect $H$-tiling exists only if $v(G)$ is divisible by $v(H)$. For a positive integer $n$ divisible by $v(H)$, we denote by $\delta(n, H)$ the minimum integer $k$ such that any $n$-vertex graph $G$ with a minimum degree of at least $k$ has a perfect $H$-tiling.
The celebrated Hajnal-Szemerédi~\cite{hajnal1970proof} theorem states that for any integer $r \geq 2$, the number $\delta(n, K_r)$ is equal to $\left(1 - \frac{1}{r}\right)n$.

An asymptotic version of the Hajnal-Szemerédi theorem for a general graph $H$ was first proven by Alon and Yuster~\cite{alon1996h}. They showed that if $n$ is divisible by $v(H)$, then $\delta(n, H) \leq \left( 1 - \frac{1}{\chi(H)}\right)n + o(n)$, where $\chi(H)$ is the chromatic number of $H$.
Komlós, Sárkőzy, and Szemerédi~\cite{komlos2001proof} improved the $o(n)$ term in the Alon-Yuster theorem to some constant $C = C(H)$, which settled the conjecture of Alon and Yuster~\cite{alon1996h}.
Another direction for an asymptotic extension of the Hajnal-Szemerédi theorem was proven by Komlós~\cite{komlos2000tiling}. Komlós showed that for any $\gamma > 0$, there exists $n_0 = n_0(\gamma, H)$ such that if $n \geq n_0$, then any $n$-vertex graph $G$ whose minimum degree is at least $\left(1 - \frac{1}{\chi_{cr}(H)} \right)n$ contains an $H$-tiling that covers at least $(1 - \gamma)n$ vertices of $G$.
Here, $\chi_{cr}(H)$ denotes the \emph{critical chromatic number} of $H$, which is defined as $\chi_{cr}(H) = \frac{(\chi(H) - 1)v(H)}{v(H) - \sigma(H)}$, where $\sigma(H)$ is the minimum possible size of a color class among all colorings of $H$ with $\chi(H)$ colors. Komlós~\cite{komlos2000tiling} conjectured that the number of uncovered vertices can be reduced to a constant, and this conjecture was confirmed by Shokoufandeh and Zhao~\cite{shokoufandeh2003proof}. More precisely, the following holds.

\begin{theorem}[Shokoufandeh and Zhao~\cite{shokoufandeh2003proof}]\label{thm:chritical-constant-missing}
    Let $H$ be a graph. Then there exists a constant $C = C(H)$, which only depends on $H$, such that any graph $G$ on $n$ vertices with a minimum degree of at least $\left(1 - \frac{1}{\chi_{cr}(H)} \right)n$ contains an $H$-tiling that covers all but at most $C$ vertices of $G$.
\end{theorem}

The almost exact value of $\delta(n, H)$ for every graph $H$ was determined by Kühn and Osthus~\cite{kuhn2009minimum} up to an additive constant depending only on $H$. To formulate this theorem, we need to define some notations.
We say a proper coloring of $H$ is \emph{optimal} if it uses exactly $\chi(H)$ colors. We denote by $\Phi(H)$ the set of all optimal colorings of $H$. Let $\phi \in \Phi(H)$ be an optimal coloring of $H$ with the size of color classes being $x_1 \leq \dots \leq x_{\chi(H)}$. Let $\calD(\phi) \defeq \{x_{i+1} - x_i : i \in [\chi(H) - 1]\}$ and $\calD(H) \defeq \bigcup_{\phi \in \Phi(H)} \calD(\phi)$.
We write $\hcf_{\chi}(H)$ for the highest common factor of all integers in $\calD(H)$. (If $\calD(H) = \{0\}$, we define $\hcf_{\chi}(H) = \infty.$) We denote by $\hcf_{c}(H)$ the highest common factor of all the orders of components in $H$.
For a graph $H$, we define
$$
\hcf(H) \defeq 
\begin{cases}
    1 & \text{if } \chi(H) \geq 3 \text{ and } \hcf_{\chi}(H) = 1,\\
    1 & \text{if } \chi(H) = 2, \text{ } \hcf_{\chi}(H) \leq 2, \text{ and } \hcf_{c}(H) = 1, \\
    0 & \text{otherwise.}
\end{cases}
$$
We define $\chi^*(H) = \chi_{cr}(H)$ if $\hcf(H) = 1$; otherwise, $\chi^*(H) = \chi(H).$ We are now ready to state the Kühn-Osthus theorem.

\begin{theorem}[Kühn and Osthus~\cite{kuhn2009minimum}]\label{thm:kuhn-osthus}
    Let $H$ be a graph and $n$ be a positive integer divisible by $v(H).$ Then there exists a constant $C = C(H)$ depending only on $H$ such that 
    \begin{equation*}
        \left(1 - \frac{1}{\chi^*(H)}\right)n - 1 \leq \delta(n, H) \leq \left(1 - \frac{1}{\chi^*(H)}\right)n + C.
    \end{equation*}
\end{theorem}

\subsection{Main results}

Motivated by the Kühn-Osthus theorem on perfect $H$-tilings, several variations of \Cref{thm:kuhn-osthus} have been considered. (For instance, see~\cite{balogh2022tilings, freschi2022dirac, han2021tilings, hurley2022sufficient, hyde2019degree, kuhn2009embedding, lo2015f}).

In this article, we consider a problem related to the concept of perfect $H$-tilings and subdivision embeddings. A graph $H'$ is a \emph{subdivision} of $H$ if $H'$ is obtained from $H$ by replacing edges of $H$ with vertex-disjoint paths while maintaining their endpoints. Since subdivisions maintain the topological structure of the original graph, various problems related to subdivisions have been posed and extensively studied. In particular, finding a sufficient condition on the host graph to contain a subdivision of $H$ is an important problem in extremal graph theory.

Let $H$ and $G$ be graphs. An \emph{$H$-subdivision tiling} is a collection of vertex-disjoint unions of subdivisions of $H$. We say that $G$ has a \emph{perfect $H$-subdivision tiling} if $G$ has an $H$-subdivision tiling that covers all vertices of $G$.
A natural question would be to determine the minimum degree threshold for a positive integer $n$ that ensures the existence of a perfect $H$-subdivision tiling in any $n$-vertex graph $G$. We define this minimum degree threshold as follows.

\begin{definition}
    Let $H$ be a graph. We denote the minimum degree threshold for perfect $H$-subdivision tilings by $\delta_{\mathrm{sub}}(n, H)$, which is the smallest integer $k$ such that any $n$-vertex graph $G$ with a minimum degree of at least $k$ has a perfect $H$-subdivision tiling.
\end{definition}

If $H$ has no edges, then a perfect $H$-subdivision tiling exists if and only if $v(G)$ is divisible by $v(H)$, regardless of the minimum degree $\delta(G)$ of $G$. Thus, from now on, we only consider graphs with at least one edge.

For most results on embedding large graphs, including perfect $H$-tilings and large subdivisions, embedding a graph with more complex structures requires a larger minimum degree threshold. Surprisingly, we observe the opposite phenomenon for perfect subdivision tiling problems. For example, \Cref{thm:main} states that $\delta_{\mathrm{sub}}(n, K_r)$ decreases as $r$ increases when $3 \leq r \leq 5$. This is due to the fact that a larger value of $r$ ensures that $K_r$ has a bipartite subdivision with a more asymmetric bipartition.

Since embedding bipartite graphs generally requires a lower minimum degree than non-bipartite graphs, we want to cover most of the vertices of the host graph with subdivisions of $H$ that are bipartite. Suppose every bipartite subdivision of $H$ is, in some sense, balanced. In that case, it cannot be tiled in a highly unbalanced complete bipartite graph that has a lower minimum degree than a balanced complete bipartite graph. For this reason, we need to measure how unbalanced bipartite subdivisions of $H$ can be, as it poses some space barriers to the problem.

For this purpose, we introduce the following two definitions.

\begin{definition}
    Let $H$ be a graph and $X \subseteq V(H).$ We define a function $f_H: 2^{V(H)} \to \mathbb{R}$ as $f_H(X) = \frac{v(H) + e(H[X]) + e(H[Y])}{|X| + e(H[Y])}$ where $Y = V(H) \setminus X.$
\end{definition}

\begin{definition}
    Let $H$ be a graph. We define $\xi(H) \defeq \min \{f_H(X) : X \subseteq V(H)\}.$
\end{definition}

Note that subdividing all edges in $H[X]$ and $H[Y]$ yields a bipartite subdivision $H'$ of $H$, and $f_H(X)$ measures the ratio between $V(H')$ and one part of its bipartition. The smaller the value of $f_H$, the more unbalanced $H'$ is. Later in \Cref{obs:why_I_define_H*}, we will prove that $\xi(H)$ can be used to measure the ratio between the two parts in the most asymmetric bipartite subdivision of $H$. Note that we always have $1 < \xi(H) \leq 2$. Indeed,
$$ f_H(V(H)) \geq \frac{v(H) + e(H)}{v(H)} > 1, $$
and for any partition $X, Y$ of $V(G)$, we have
$$ v(H) + e(H[X]) + e(H[Y]) = |X| + e(H[X]) + |Y| + e(H[Y]) \leq 2 \max\{|Z| + e(H[Z]) : Z \in \{X, Y\}\}, $$
so, $\min\{f_H(X), f_H(Y)\} \leq 2$.

Another crucial factor is the divisibility issue. Assume that all bipartite subdivisions of $H$ have bipartitions with both parts having the same parity. If $G$ is a complete bipartite graph $K_{a, b}$ with $a$ and $b$ having different parity, then we cannot find perfect $H$-subdivision tilings in $G$, as it poses some divisibility barriers to the problem. Hence, we need to introduce the following definitions concerning the difference between the two parts in bipartitions of subdivisions of $H$ and their highest common factor.

\begin{definition}
    Let $H$ be a graph. We define $\calC(H) \defeq \{(|X| + e(H[Y])) - (|Y| + e(H[X])) : X \subseteq V(H), Y = V(H) \setminus X\}$. We denote by $\hcf_{\xi}(H)$ the highest common factor of all integers in $\calC(H)$. (If $\calC(H) = \{0\}$, we define $\hcf_{\xi}(H) = \infty.$)
\end{definition}

By considering the space barrier and the divisibility barrier, we introduce the following parameter measuring both obstacles for the problem. We will show that this is the determining factor for $\delta_{\mathrm{sub}}(n, H)$.

\begin{definition}
    Let $H$ be a graph. We define
    $$
    \xi^*(H) \defeq 
    \begin{cases}
        \xi(H) &\text{if } \hcf_{\xi}(H) = 1,\\
        \max\left\{\frac{3}{2}, \xi(H)\right\} &\text{if } \hcf_{\xi}(H) = 2,\\
        2 &\text{otherwise.}
    \end{cases}
    $$
\end{definition}

We are now ready to state our main theorems. The following theorem gives the asymptotically exact value for $\delta_{\mathrm{sub}}(n, H)$ except in one case, when $\hcf_{\xi}(H) = 2$.

\begin{theorem}\label{thm:main}
    Let $H$ be a graph with $\hcf_{\xi}(H) \neq 2$. For every $\gamma > 0$, there exists an integer $n_0 = n_0(\gamma, H)$ such that for any $n \geq n_0$, the following holds.
    \begin{equation*}
        \left(1 - \frac{1}{\xi^*(H)} \right)n - 1 \leq \delta_{\mathrm{sub}}(n, H) \leq \left(1 - \frac{1}{\xi^*(H)} + \gamma \right)n
    \end{equation*}
\end{theorem}

This theorem asymptotically determines $\delta_{\mathrm{sub}}(n, H)$ as long as $\hcf_{\xi}(H) \neq 2$. If $\hcf_{\xi}(H) = 2$, then the parity of $n$ is also important. The following theorem asymptotically determines $\delta_{\mathrm{sub}}(n, H)$ for this case.

\begin{theorem}\label{thm:hcf-2}
    Let $H$ be a graph with $\hcf_{\xi}(H) = 2$. For every $\gamma > 0$, there exists an integer $n_0 = n_0(\gamma, H)$ such that the following holds. For every integer $n \geq n_0$,
    \begin{align*}
        \frac{1}{2}n - 1 &\leq \delta_{\mathrm{sub}}(n, H) \leq \left(\frac{1}{2} + \gamma \right)n &&\text{if } n \text{ is odd,}\\ 
        \left(1 - \frac{1}{\xi^*(H)} \right)n - 1 &\leq \delta_{\mathrm{sub}}(n, H) \leq \left(1 - \frac{1}{\xi^*(H)} + \gamma \right)n &&\text{if } n \text{ is even.}
    \end{align*}
\end{theorem}

As a direct consequence of \Cref{thm:main,thm:hcf-2}, we determine the value of $\delta_{\mathrm{sub}}(n, K_r)$ for each $r \geq 2.$ By applying our main theorems directly, we have $\delta_{\mathrm{sub}}(n, K_2) = \left(\frac{1}{3} + o(1) \right)n$, and for each $r \in \{3, 4, 5\}$, we have $\delta_{\mathrm{sub}}(n, K_r) = \left(\frac{2}{r+1} + o(1) \right)n$. For the case $r = 7$, if $n$ is even, we have $\delta_{\mathrm{sub}}(n, K_7) = \left(\frac{1}{3} + o(1) \right)n$; otherwise, we have $\delta_{\mathrm{sub}}(n, K_7) = \left(\frac{1}{2} + o(1) \right)n$. Finally, for every $r \geq 8$ and $r = 6$, we have $\delta_{\mathrm{sub}}(n, K_r) = \left(\frac{1}{2} + o(1) \right)n$. This is in contrast to the normal $H$-tiling problem.

This means the determining factors for the minimum degree thresholds of perfect $H$-tilings and perfect $H$-subdivision tilings are essentially different. Probably, the most interesting difference between the perfect $H$-tiling and the perfect $H$-subdivision tiling is that monotonicity does not hold for subdivision tiling. For a perfect tiling, if $H_2$ is a spanning subgraph of $H_1$, then obviously $\delta(n, H_2) \leq \delta(n, H_1).$ However, for perfect subdivision tiling, this does not hold in many cases. For example, our results imply $\delta_{\mathrm{sub}}(n, K_4) = \frac{2}{5}n + o(n) < \delta_{\mathrm{sub}}(n, C_4) = \frac{1}{2}n + o(n).$

As $\xi^*(H)$ is the determining factor for the minimum degree threshold, it is convenient for us to specify the bipartite subdivision achieving the value $\xi^*(H).$ We introduce the following definition.

\begin{definition}
    Let $H$ be a graph. We denote by $X_H$ the subset of $V(H)$, where $f_H(X_H) = \xi(H).$ We define a graph $H^*$ obtained from $H$ by replacing all edges in $H[X_H]$ and $H[V(H) \setminus X_H]$ with paths of length two.
\end{definition}

There can be multiple choices for $X_H$. Then we fix one choice for $X_H$ so that $X_H$ and $H^*$ are uniquely determined for all $H$. Note that $H^*$ is a subdivision of $H$, which is a bipartite graph, and $v(H^*) = v(H) + e(H[X_H]) + e(H[V(H) \setminus X_H])$. Intuitively, $\xi(H)$ seems greater than $\chi_{cr}(H^*)$. Indeed, this intuition is correct. However, \Cref{thm:main} cannot be directly deduced by applying \Cref{thm:kuhn-osthus} for $H^*$. For example, for a connected graph $H$, it may satisfy $\chi^*(H^*) = 2$, but $\xi^*(H) < 2$ is possible. Then the minimum degree threshold in \Cref{thm:main} is smaller than the one in \Cref{thm:kuhn-osthus} when $\hcf_{\xi}(H) = 1$.

However, as we will see in \Cref{obs:xi-chicr}, for every graph $H$ with at least one edge, the inequality $\xi(H) \geq \chi_{cr}(H^*)$ holds. Hence, if $\hcf_{\xi}(H) \leq 2$, we may instead use \Cref{thm:chritical-constant-missing} to find an $H^*$-tiling that covers all but at most a constant number of vertices of $G$ in a graph $G$ with $\delta(G) \geq \left(1 - \frac{1}{\xi(H)} + \gamma \right)n$. To cover the leftover vertices, we use the absorption method. The absorption method was introduced in~\cite{rodl2006dirac}, and since then, it has been used to solve various crucial problems in extremal combinatorics. The main difficulty in applying the absorption method in our setting is that, in many cases, the host graph is not sufficiently dense to guarantee that any vertices can be absorbed in the final step. To overcome this difficulty, we use the regularity lemma and an extremal result on the domination number to obtain some control over the vertices that can be absorbed.       
    
\section{Preliminaries}\label{sec:prelim}

We write $[n] = \{1, \dots , n\}$ for a positive integer $n$. If we claim a result holds if $\beta \ll \alpha_1, \dots, \alpha_t$, then it means there exists a function $f$ such that $\beta \leq f(\alpha_1, \dots, \alpha_t)$. We will not explicitly compute these functions. In this paper, we consider $o(1)$ to go to zero as $n$ goes to infinity.

Let $G$ be a graph. We denote the vertex set and edge set of $G$ by $V(G)$ and $E(G)$, respectively, and we set $v(G) = |V(G)|$ and $e(G) = |E(G)|$. We write $d_G(v)$ for the degree of $v \in V(G)$, and we omit the subscript if the graph $G$ is clear from the context. We denote by $\delta(G)$ and $\Delta(G)$ the minimum degree of $G$ and the maximum degree of $G$, respectively. For a vertex $v \in V(G)$, we denote by $N_G(v)$ the set of neighbors of $v$ in $G$ which are adjacent to $v$. We also omit the subscript if $G$ is clear from the context.

For a graph $G$ and a vertex subset $X \subseteq V(G)$, we denote by $G[X]$ the subgraph of $G$ induced by $X$. For vertex subsets $A$ and $B$ of $G$, we denote by $G[A, B]$ the graph where $V(G[A, B]) = A \cup B$ and $E(G[A, B]) = \{uv \in E(G) : u \in A, v \in B\}$. 
Let $v$ be a vertex of $G$ and $X$ be a vertex subset of $G$. We denote by $d_G(v; X)$ the degree of $v$ in the induced graph $G[X \cup \{v\}]$.

Let $G$ be a graph and $X \subseteq V(G)$. We denote by $G - X$ the induced graph $G[V(G) \setminus X]$. If $X$ is a single vertex $v$, we simply denote $G - v$. Similarly, if a graph $H$ is a subgraph of $G$, then we denote by $G - H$ the induced subgraph $G[V(G) \setminus V(H)]$. Let $G'$ be a subgraph of $G$ and $v \in V(G)$. We denote by $G' + v$ a graph such that $V(G' + v) = V(G') \cup \{v\}$ and $E(G' + v) = E(G') \cup E(G[\{v\}, V(G')])$.

We denote by $K_r$ the complete graph of order $r$ and $K_{n, m}$ the complete bipartite graph with bipartition of sizes $n$ and $m$. We also denote by $C_k$ and $P_k$ the cycle of length $k$ and the path of length $k$, respectively. 
We say $H'$ is a \emph{1-subdivision} of $H$ if $H'$ is obtained from $H$ by replacing all edges of $H$ with vertex-disjoint paths of length two. We denote by $H^1$ the 1-subdivision of $H$. We write $\mathrm{Sub}(H)$ for the collection of all subdivisions of $H$. For a graph $F \in \mathrm{Sub}(H)$, we say a vertex $v \in V(F) \cap V(H)$ is a \emph{branch vertex} of $F$.

In the proofs of the main theorems, we need to count the number of copies of specific bipartite subgraphs where each bipartition is contained in certain vertex subsets. In order to facilitate this counting, we introduce the notion of embedding. Let $H$ and $G$ be graphs. An \emph{embedding} $\phi$ of $H$ into $G$ is an injective function from $V(H)$ to $V(G)$ such that for any $uv \in E(H)$, its image $\phi(u) \phi(v)$ is in $E(G)$.

The next simple observation will be used in \Cref{sec:absorber}.
\begin{observation}\label{obs:delet-small-vertices}
    Let $0 \leq \gamma, \frac{1}{n} \ll \frac{1}{t}, d < 1$.
    Let $\calB$ be an $n$-vertex $t$-uniform hypergraph with $|E(\calB)| \geq dn^t$. Let $A \subseteq V(\calB)$ be a vertex subset of size at most $\gamma n$. Then we have $|E(\calB - A)| \geq \frac{d}{2}n^t$.
\end{observation}
As every vertex is contained in at most $n^{t - 1}$ edges, by choosing $\gamma$ sufficiently small, we can easily verify that \Cref{obs:delet-small-vertices} holds.
In the following two subsections, we introduce powerful tools in extremal graph theory, so-called supersaturation, and the regularity lemma.

\subsection{The supersaturation}\label{subsec:supersaturation}

One of the most fundamental problems in extremal graph theory is to determine the maximum number of edges in an $n$-vertex graph that does not contain a copy of a specific graph as a subgraph. For a graph $H$ and a positive integer $n$, we denote by $\mathrm{ex}(n, H)$ the \emph{extremal number} of $H$, defined as the maximum integer $k$ such that there is an $n$-vertex graph $G$ with $k$ edges not containing $H$ as a subgraph.

A classical theorem of Turán determined the exact value of $\mathrm{ex}(n, K_r)$. For a graph $H$, we define the \emph{Turán density} of $H$ as $\pi(H) \defeq \lim_{n \to \infty} \frac{\mathrm{ex}(n, H)}{\binom{n}{2}}$. The existence of Turán density for every graph $H$ was proved by Katona, Nemetz, and Simonovits~\cite{katona1964graph}. For every graph $H$, the Erdős-Stone-Simonovits~\cite{erdHos1965limit, erdos1946structure} theorem states that $\pi(H) = \left(1 - \frac{1}{\chi(H)-1}\right)$. We note that if a graph $H$ is bipartite, the Turán density satisfies $\pi(H) = 0$.

Assume that an $n$-vertex graph $G$ has many more edges than $\mathrm{ex}(n, H)$. Intuitively, we can expect there to be a lot of copies of $H$ in $G$. Indeed, this is true, and this phenomenon is called \emph{supersaturation}, proved by Erdős and Simonovits~\cite{erdHos1983supersaturated}, as follows.

\begin{theorem}[Supersaturation]\label{thm:supersaturation}
    Let $0 < \delta \ll \varepsilon, \frac{1}{h} \leq 1$.
    For every graph $H$ on $h$ vertices, the following holds. If $G$ is an $n$-vertex graph with $e(G) \geq (\pi(H) + \varepsilon)\binom{n}{2}$, then $G$ contains at least $\delta n^h$ copies of $H$.
\end{theorem}

Since the Turán density of a bipartite graph $H$ is zero, \Cref{thm:supersaturation} leads to the following lemmas.

\begin{lemma}\label{lem:Kab-supersaturation}
    Let $0 < \delta \ll \frac{1}{a}, \frac{1}{b}, \varepsilon \leq 1$ and $n$ be a positive integer. For every complete bipartite graph $H$ on bipartition $A$ and $B$ with $|A| = a$ and $|B| = b$, the following holds. Let $G$ be a bipartite graph with bipartition $X$ and $Y$ such that $|X|, |Y| \leq n$. If $e(G) \geq \varepsilon n^2$, then the number of embeddings $\phi : V(H) \to V(G)$ is at least $\delta n^{a+b}$, where $\phi(A) \subseteq X$ and $\phi(B) \subseteq Y$.
\end{lemma}

\begin{proof}
    Assume $a \leq b$ and let $v(G) = n'$. Since $e(G) \geq \varepsilon n^2$, we observe that the inequality $\sqrt{2 \varepsilon} n \leq n' = |A| + |B| \leq 2n$ holds. Thus, if we choose $n$ sufficiently large, then also $n'$ is sufficiently large, and we have $e(G) \geq \varepsilon n'^2$. By \Cref{thm:supersaturation}, the number of copies of $K_{b, b}$ in $G$ is at least $\delta' n'^{2b}$ for some $\delta' > 0$. We observe that for each copy of $H$ in $G$, where $A$ is embedded in $X$, there are at most $(n')^{b - a}$ copies of $K_{b, b}$ in $G$. Thus, by double counting, the number of embeddings of $H$ in $G$, where $A$ is embedded in $X$ and $B$ is embedded in $Y$, is at least $\delta' (n')^{a + b} \geq \delta n^{a + b}$.
\end{proof}

\begin{lemma}\label{lem:fix-embedding}
    Let $0 < d \ll \varepsilon, \frac{1}{h} \leq 1$. For every $h$-vertex bipartite graph $H$ on bipartition $A$ and $B$, the following holds.
    Let $G$ be an $n$-vertex graph with minimum degree at least $\varepsilon n$. Let $X \subseteq V(G)$ be a set of at least $\varepsilon n$ vertices. 
    Then there are at least $\delta n^{h}$ distinct embeddings $\phi : V(H) \to V(G)$, where $\phi(A) \subseteq X$.
\end{lemma}

\begin{proof}
    Let $X' \defeq V(G) \setminus X$. If $e(G[X]) \geq \frac{\varepsilon^2}{4} n^2$, by \Cref{thm:supersaturation}, the induced graph $G[X]$ contains at least $\delta n^{h}$ distinct copies of $H$. Thus, in this case, the lemma is proved. We now assume $e(G[X]) \leq \frac{\varepsilon^2}{4} n^2$. Then, by the minimum degree condition of $G$, the number of edges in the bipartite subgraph $G[X, X']$ is at least $\frac{\varepsilon^2}{4}n^2$. Then, by \Cref{lem:Kab-supersaturation}, for the bipartite graph $G[X, X']$, there are at least $\delta n^{h}$ distinct desired embeddings. This proves the lemma. 
\end{proof}

\subsection{The regularity lemma}\label{subsec:regularity}

Szemerédi's regularity lemma~\cite{szemeredi1975regular} is a powerful tool for dealing with large dense graphs. To formulate the regularity lemma, we need to define an $\ve$-regular pair. Let $G$ be a graph and $A, B \subseteq G$. We define the density between $A$ and $B$ as $d_G(A, B) \defeq \frac{e(G[A, B])}{|A||B|}$. The following is the definition of an $\ve$-regular pair.

\begin{definition}
    Let $\ve > 0$ and $d \in [0, 1]$. A pair $(A, B)$ of disjoint subsets of vertices in a graph $G$ is an \emph{$(\ve, d)$-regular pair} if the following holds. For any subset $A' \subseteq A$, $B' \subseteq B$ with $|A'| \geq \ve |A|$ and $|B'| \geq \ve |B|$,
    \begin{equation*}
        |d_G(A', B') - d| \leq \ve.
    \end{equation*}
    If a pair $(A, B)$ is $(\ve, d)$-regular for some $d \in [0, 1]$, we say $(A, B)$ is an $\ve$-regular pair. We also say a pair $(A, B)$ is $(\ve, d+)$-regular if $(A, B)$ is $(\ve, d')$-regular for some $d' \geq d$.
\end{definition}

The following is the degree form of the regularity lemma~\cite[Theorem 1.10]{komlos1996szemeredi}.

\begin{lemma}[Regularity Lemma-degree form]\label{lem:regularity}
    Let $0 < \frac{1}{T_0} \ll \frac{1}{t_0}, \ve \leq 1$. For every real number $d \in [0, 1]$ and every graph $G$ with order at least $T_0$, there exists a partition of $V(G)$ into $t+1$ clusters $V_0, V_1, \dots , V_t$ and a spanning subgraph $G_0$ of $G$ such that the following holds:
    
    \begin{itemize}
        \item[$\bullet$] $t_0 \leq t \leq T_0$,
        \item[$\bullet$] $|V_0| \leq \ve v(G)$,
        \item[$\bullet$] $|V_1| = \cdots = |V_t|$,
        \item[$\bullet$] $G_0[V_i]$ has no edge for each $i \in [t]$,
        \item[$\bullet$] $d_{G_0}(v) \geq d_G(v) - (d + \ve)v(G)$ for all $v \in V(G)$,
        \item[$\bullet$] for all $1 \leq i < j \leq t$, either the pair $(V_i, V_j)$ is an $(\ve, d+)$-regular pair or the graph $G_0[A, B]$ has no edge.
    \end{itemize}
\end{lemma}

The partition in \Cref{lem:regularity} is called an \emph{$(\ve, d)$-regular partition} or simply an \emph{$\ve$-regular partition}. We now define an \emph{$(\ve, d)$-reduced graph} $R$ on the vertex set $\{V_1, \dots , V_t\}$, which is obtained from \Cref{lem:regularity}, by joining the edges for each $V_i, V_j$ if and only if $(V_i, V_j)$ is an $(\ve, d+)$-regular pair. The following lemma shows that the reduced graph inherits the minimum degree condition~\cite{kuhn2005large}.

\begin{lemma}[\cite{kuhn2005large}, Proposition 9]\label{lem:reduced-minimum-degree}
    For every $\gamma > 0$, there exist $\ve_0 = \ve_0(\gamma)$ and $d_0 = d_0(\gamma) > 0$ such that for every $\ve \leq \ve_0$, $d \leq d_0$, and $\delta > 0$, for every graph $G$ with minimum degree $(\delta + \gamma)n$, the reduced graph $R$ on $G$ obtained by applying \Cref{lem:regularity} with parameters $\ve$ and $d$, the minimum degree of $R$ is at least $\left(\delta + \frac{\gamma}{2}\right)|R|$.
\end{lemma}

Reduced graphs are useful for analyzing the approximate structures of the host graphs, as they inherit many properties of the host graphs, such as the minimum degree.

The main reason that we need the regularity lemma is to get an efficient absorber for the leftover vertices at the final step of the proofs. To achieve this, we need an additional minimum degree condition for the $\ve$-regular pair. The following lemma guarantees that we can delete low-degree vertices from $A$ and $B$ while preserving the $\ve$-regularity.

\begin{lemma}[\cite{bottcher2009proof}, Proposition 12]\label{lem:super-regular-delete}
    Let $(A, B)$ be an $(\ve, d)$-regular pair in a graph $G.$ Let $B' \subseteq B$ with $|B'| \geq \ve |B|.$ Then the size of the set $\{a \in A : d_G(a; B') < (d - \ve)|B'|\}$ is at most $\ve |A|$. 
\end{lemma}

To get more information on the regularity lemma, see the following papers~\cite{kim2019blow, komlos1997blow, komlos2002regularity, komlos1996szemeredi, rodl2010regularity}.

\section{Extremal examples and properties of $\xi(H)$ and $\chi_{cr}(H^*)$}\label{sec:ext-prop}

We first show that for any $\gamma > 0$ and every graph $H$, there is $n_0 = n_0(\gamma, H)$ such that for all $n \geq n_0$, the inequality $\delta_{\mathrm{sub}}(n, H) \leq \left(\frac{1}{2} + \gamma \right)n$ holds. Note that for any graph $H$, we have $\left(1 - \frac{1}{\chi^*(H^*)} \right) \leq \frac{1}{2}$ since $H^*$ is bipartite. Then \Cref{thm:kuhn-osthus} implies $\delta(n, H^*) \leq \left(\frac{1}{2} + \frac{\gamma}{2} \right)n$ if $n$ is sufficiently large and divisible by $v(H^*)$.

\begin{proposition}\label{prop:not-exceed-1/2}
    Let $0 < \frac{1}{n} \ll \gamma , \frac{1}{h} \leq 1.$ Then for every $h$-vertex graph $H$, the following holds:
    $$
    \delta_{\mathrm{sub}}(n, H) \leq \left(\frac{1}{2} + \gamma \right)n.
    $$
\end{proposition}

\begin{proof}
    Let $G$ be an $n$-vertex graph and $\delta(G) \geq \left(\frac{1}{2} + \gamma \right)n.$ By subdividing one edge of $H^*$ at most $v(H^*)$ more times, we obtain a graph $F$ with a constant number of vertices, $\chi(F) \leq 3$, and $v(G) - v(F)$ is divisible by $v(H^*)$.
    By the Erdős-Stone-Simonovits theorem, we can find a copy of $F$ in $G.$
    Since $v(F)$ is bounded, the minimum degree of $G[V(G)\setminus V(F)]$ is greater than $\left(\frac{1}{2} + \frac{\gamma}{2}\right)n$ for $n \geq n_0$, where $n_0$ depends only on $H.$ 
    Then by \Cref{thm:kuhn-osthus}, there is a perfect $H^*$-tiling in $G[V(G)\setminus V(F)]$ which yields a perfect $H$-subdivision tiling in $G$ together with the copy of $F.$
    This completes the proof.
\end{proof}

We now prove lower bounds on $\delta_{\mathrm{sub}}(n, H)$ by constructing graphs without perfect $H$-subdivision tilings. The following observation shows that $H^*$ is a bipartite subdivision with the most unbalanced bipartition. We recall that $\xi(H)$ is the minimum value among all possible values of $f_H$ where $f_H: 2^{V(H)}\to \mathbb{R}$ is a function such that for every $X\subseteq V(G)$, $f_H(X) = \frac{v(H)+e(H[X])+e(H[Y])}{|X| + e(H[Y])}$, where $Y = V(G)\setminus X.$

\begin{observation}\label{obs:why_I_define_H*}
    Let $H$ be a graph and let $F\in \mathrm{Sub}(H)$ be a bipartite graph with bipartition $A$ and $B.$ Then $\frac{|B|}{|A|} \leq \frac{1}{\xi(H) - 1}.$
\end{observation}

\begin{proof}
    We may assume $|A| \leq |B|.$ It suffices to show that $\xi(H) \leq \frac{|A| + |B|}{|B|}.$
    Let $U$ be the branch vertices of $F$ and let $X = A \cap U$ and $Y = B \cap U.$ Let $p = |X| + e(H[Y])$ and $q = |Y| + e(H[X]).$ Then $f_H(Y) = \frac{p + q}{q}.$ 
    Since $F$ is a bipartite subdivision of $H$ with branch set $X \cup Y$, every edge in $H[X, Y]$ is subdivided an even number of times. Thus, there is a non-negative integer $c$ such that $\frac{|A| + |B|}{|B|} = \frac{p + q + 2c}{q + c}$.
    If $\frac{p + q}{q} \geq 2$, then $\frac{|A| + |B|}{|B|} \geq 2.$ Since $\xi(H) \leq 2$, in this case, the observation is proved. Otherwise, we may observe that $p < q.$ Then $\frac{|A| + |B|}{|B|} = \frac{p + q + 2c}{q + c} \geq \frac{p + q}{q}$ holds. By the definition of $\xi(H)$, we have $\xi(H) \leq H(Y) = \frac{p + q}{q}.$ This completes the proof.
\end{proof}

By using \Cref{obs:why_I_define_H*}, we can show the following proposition.

\begin{proposition}\label{prop:lowerbound-hcf}
    For every integer $n > 0$ and every graph $H$, there is an $n$-vertex graph $G$ with minimum degree at least $\left\lfloor\left(1 - \frac{1}{\xi(H)} \right)n\right\rfloor - 1$ such that $G$ does not have a perfect $H$-subdivision tiling.
\end{proposition}

\begin{proof}
    Let $G$ be an $n$-vertex complete bipartite graph with bipartition $X$ and $Y$ such that $|X| \leq |Y|$. Note that as $G$ is a bipartite graph, every subdivision of $H$ in $G$ is bipartite. By \Cref{obs:why_I_define_H*}, if $\frac{|Y|}{|X|} > \frac{1}{\xi(H) - 1}$, then $G$ does not have a perfect $H$-subdivision tiling. Let $|X| = \left\lfloor \left(1 - \frac{1}{\xi(H)} \right)n \right\rfloor - 1$. Then $\delta(G) = \left\lfloor \left(1 - \frac{1}{\xi(H)} \right)n \right\rfloor - 1$ and $\frac{|Y|}{|X|} > \frac{1}{\xi(H) - 1}$. Thus, $G$ does not have a perfect $H$-subdivision tiling.
\end{proof}

The next proposition provides a reason for why $\hcf_{\xi}(H)$ is the determining factor of the value $\xi^*(H)$.

\begin{proposition}\label{prop:hcf-not-2}
    Let $H$ be a graph with $\hcf_{\xi}(H) \neq 1$. Then for every integer $n > 0$, there is an $n$-vertex graph $G$ with minimum degree at least $\left\lfloor \frac{n}{2} \right\rfloor - 1$ which does not have a perfect $H$-subdivision tiling except for $\hcf_{\xi}(H) = 2$ and $n$ is even.
\end{proposition}

To prove \Cref{prop:hcf-not-2}, we need the following claim.
\begin{claim}\label{clm:hcf-bipartite}
    Let $H$ be a graph with at least one edge and let $H'$ be a bipartite subdivision of $H$ with bipartition $A$ and $B$. Then $|B| - |A|$ is divisible by $\hcf_{\xi}(H)$.
\end{claim}

\begin{claimproof}[Proof of \Cref{clm:hcf-bipartite}]
    Let $A'$ and $B'$ be the branch vertices of $H$ contained in $A$ and $B$, respectively. Let $uv\in E(H)$ be an edge and let $u z_1, \dots, z_k, v$ be the subdivided path in $H'$ that corresponds to $uv$. If $u\in A$ and $v\in B$, then $k$ is an even number since $H'$ is a bipartite graph. Thus, we observe that $|B\cap \{z_1,\dots, z_k\}| - |A\cap \{z_1, \dots, z_k\}| = 0$. If both $u$ and $v$ are in $A$, the number $k$ is odd and the equality $|B\cap \{z_1,\dots, z_k\}| - |A\cap \{z_1, \dots, z_k\}| = 1$ holds since $H'$ is a bipartite graph. Similarly, if both $u$ and $v$ are in $B$, then we have $|B\cap \{z_1,\dots, z_k\}| - |A\cap \{z_1, \dots, z_k\}| = -1$. Hence, we have the following. 
    $$|B| - |A| = e(H[A']) - e(H[B']) + |B'| - |A'| = (|B'| + e(H[A'])) - (|A'| + e(H[B'])).$$ 
    Thus, by the definition of $\hcf_{\xi}(H)$, the number $|B| - |A|$ is divisible by $\hcf_{\xi}(H)$. This completes the proof.
\end{claimproof}

\begin{proof}[Proof of \Cref{prop:hcf-not-2}]
    Let $F\in \mathrm{Sub}(H)$ be a bipartite graph with bipartition $A$ and $B$. Since $F$ is a bipartite subdivision of $H$, by \Cref{clm:hcf-bipartite}, the difference $|B| - |A|$ is divisible by $\hcf_{\xi}(H)$. If $G$ is a bipartite graph, and the difference between the sizes of two bipartitions of $G$ is not divisible by $\hcf_{\xi}(H)$, the graph $G$ does not have a perfect $H$-subdivision tiling since every subdivision of $H$ in $G$ is bipartite.
    
    First, let $n$ be an even number and $\hcf_{\xi}(H) > 2$. Let $G = K_{\frac{n}{2} - 1, \frac{n}{2} + 1}$. Then, the difference between the two bipartitions of $G$ is two which cannot be divisible by $\hcf_{\xi}(H)$ since $\hcf_{\xi}(H) > 2$.

    Second, let $n$ be an odd number and $\hcf_{\xi}(H) \neq 1$. Let $G = K_{\left\lfloor \frac{n}{2} \right\rfloor, \left\lceil \frac{n}{2} \right\rceil}$. Then, the difference between the sizes of the two bipartitions of $G$ is $1$, which is not divisible by $\hcf_{\xi}(H)$ as $\hcf_{\xi}(H) \neq 1$.
\end{proof}

In the following proposition, we consider the case when $\hcf_{\xi}(H) = 2$ and the order of the host graph is even.

\begin{proposition}\label{prop:hcf-2-1/3}
    For every graph $H$ with $\hcf_{\xi}(H) = 2$ and for every even number $n$, there is an $n$-vertex graph $G$ with minimum degree at least $\left\lfloor \frac{1}{3}n \right\rfloor - 1$ such that $G$ does not contain a perfect $H$-subdivision tiling.    
\end{proposition}

\begin{proof}
    Let $(A, B, C)$ be a partition of $[n]$ with sizes $|A| = \left\lfloor \frac{1}{3}n \right\rfloor - 1$ and $|B| = \left\lfloor \frac{1}{3}n \right\rfloor$. Let $G$ be a graph on vertex set $[n]$ such that $G$ is a disjoint union of a clique on $C$ and a complete bipartite graph with bipartition $(A, B)$. Then, the minimum degree of $G$ is $\left\lfloor \frac{1}{3}n \right\rfloor - 1$. Let $G_1 = G[A \cup B]$ and $G_2 = G[C]$. Since $G_1$ and $G_2$ are disjoint, if $G$ contains a perfect $H$-subdivision tiling, so does $G_1$. However, the difference between the sizes of the bipartitions of $G_1$ is $1$, which is not divisible by $\hcf_{\xi}(H) = 2$. Thus, by \Cref{clm:hcf-bipartite}, the graph $G_1$ does not contain a perfect $H$-subdivision tiling since $G_1$ is a bipartite graph, so every subdivision of $H$ in $G$ is also bipartite. Therefore, $G$ does not have a perfect $H$-subdivision tiling.
\end{proof}

The next observation shows the relationship between $\xi(H)$ and $\chi_{cr}(H^*)$.

\begin{observation}\label{obs:xi-chicr}
    For every graph $H$, the inequality $\left(1 - \frac{1}{\xi(H)} \right) \geq \left(1 - \frac{1}{\chi_{cr}(H^*)} \right)$ holds.
\end{observation}

\begin{proof}
    It suffices to show that $\chi_{cr}(H^*) \leq \xi(H)$. We recall that $X_H$ is a subset of $V(H)$ such that $f_H(X_H) = \xi(H)$. Let $Y_H \defeq V(H)\setminus X_H$. Let $A = X_H$ and $B = Y_H$. Since $H^*$ is obtained by replacing all edges in $H[X_H]$ and $H[Y_H]$ with vertex-disjoint paths of length two, we update the set $B$ by adding vertices that were used to replace the edges in $H[X_H]$. Similarly, we update the set $A$ by adding vertices that were used to replace the edges in $H[Y_H]$. Then $A$ and $B$ are independent sets and $|B| = |Y_H| + e(H[X_H])$. Thus, by the definition of $\sigma(H^*)$, the inequality $\sigma(H^*) \leq |Y_H| + e(H[X_H])$ holds. This implies $\chi_{cr}(H^*) \leq \frac{v(H^*)}{v(H^*) - (|Y_H| + e(H[X_H]))} = \frac{v(H)+ e(H[X_H]) + e(H[Y_H])}{|X_H| + e(H[Y_H])} = \xi(H)$. 
\end{proof}

\section{$\mathrm{Sub}(H)$-absorbers}\label{sec:absorber}
We use the absorption method to prove \Cref{thm:main,thm:hcf-2}. In order to execute the absorption method, we need some specific structures that we call $\mathrm{Sub}(H)$-absorbers.

\begin{definition}
    Let $H$ and $G$ be graphs and take two subsets $A\subseteq V(G)$ and $X\subseteq V(G)\setminus A.$ We say $A$ is a \emph{$\mathrm{Sub}(H)$-absorber} for $X$ if both $G[A]$ and $G[A\cup X]$ have perfect $H$-subdivision tilings. If $X = \{v\}$, we say $A$ is a $\mathrm{Sub}(H)$-absorber for $v.$   
\end{definition}

Our strategy for constructing $\mathrm{Sub}(H)$-absorbers is to collect many vertex-disjoint small subgraphs with good properties, which we call \emph{absorber units}. We define three different types of absorber units, which we call type-$1$, type-$2$, and type-$3$ $\mathrm{Sub}(H)$-absorbers, respectively. 
We first introduce the type-$1$ $\mathrm{Sub}(H)$-absorber as follows.

\begin{definition}\label{def:type1unit}
    Let $H$ be a graph and $H^1$ be the $1$-subdivision of $H.$ Since $H^1$ is bipartite, $H^1$ has a bipartition $(U, V)$ such that the set $U$ consists of all branch vertices. Choose an edge $xy\in E(H^1)$, with $x\in U$, and $y\in V.$ Now we replace the edge $xy$ with a path of length three, say $x w z y$ and denote by $T^1_H$ the obtained graph. Let $u$ and $v$ be two new vertices and add edges $uy, uw, vx,$ and $vz$ to $T^1_H$ and we denote by $\hat{T}^1_H$ the obtained graph.

    Let $G$ be a graph and let two distinct vertices $a$ and $b\in V(G).$ We say a set $A\subseteq V(G)\setminus \{a, b\}$ is a \emph{type-$1$ $\mathrm{Sub}(H)$-absorber} for $\{a, b\}$ if $|A| = v(T^1_H)$ and there is an embedding $\phi: V(\hat{T}^1_H)\to A\cup\{a, b\}$ such that $\phi(u) = a$ and $\phi(v) = b.$
\end{definition}

We assume that the choice of $xy$ is always the same for all $H$, so that $T^1_H$ is uniquely determined for all $H.$
A graph $T^1_H$ is a subdivision of $H$ since it is obtained from the $1$-subdivision of $H$ by replacing one edge with a vertex-disjoint path of length $3.$ By the construction of $\hat{T}^1_H$, we can replace a path $xwzy$ with a path $xvzwuy$, so $\hat{T}^1_H$ also contains a spanning subdivision of $H.$ 
Thus, for a graph $G$ and $a$ and $b\in V(G)$, a set $X\subseteq V(G)$ is a type-$1$ $\mathrm{Sub}(H)$-absorber for $\{a, b\}$, implying both $G[X]$ and $G[X\cup\{a, b\}]$ have perfect $H$-subdivision tilings.

Type-$1$ $\mathrm{Sub}(H)$-absorbers are the most efficient bipartite absorber units. If the host graph is bipartite, then we only obtain bipartite absorber units. Note that a bipartite graph cannot absorb a single vertex, so absorbing two vertices is as good as it gets in a bipartite host graph.
Moreover, since $T^1_H$ is bipartite, supersaturation implies that there are many type-$1$ $\mathrm{Sub}(H)$-absorbers, allowing us to obtain many efficient $\mathrm{Sub}(H)$-absorbers.

If there are many type-$1$ $\mathrm{Sub}(H)$-absorbers for every pair of vertices in $G$, it would be ideal to obtain the desired absorbers. However, this is not true in general, so it is useful to define the following auxiliary graph that captures information on vertex pairs that have many type-$1$ $\mathrm{Sub}(H)$-absorbers. 

\begin{definition}
    Let $H$ be a graph and $d > 0$. For a given $n$-vertex graph $G$, we define the \emph{$d$-absorbing graph} $F$ as the graph on the vertex set $V(G)$ such that $uv\in E(F)$ if and only if the number of type-$1$ $\mathrm{Sub}(H)$-absorbers for $\{u, v\}$ is at least $dn^{v(T^1_H)}.$
\end{definition}

In order to deal with non-bipartite host graphs, we define the following two types of absorber units. Note that these absorbers are able to absorb one vertex.

\begin{definition}
    Let $H$ be a graph and $(U, V)$ be a bipartition of $H^1$ such that $U$ is the set of all branch vertices. Let $x, y$ be the vertices as in \Cref{def:type1unit}. For consistency, we denote by $T^2_H$ the copy of $H^1.$ Let $u$ be a new vertex and add edges $ux, uy$ to $T^2_H$ and we denote by $\hat{T}^2_H$ the obtained graph.

    Let $G$ be a graph and let $a\in V(G).$ We say a set $A\subseteq V(G)\setminus \{a\}$ is a \emph{type-$2$ $\mathrm{Sub}(H)$-absorber} for $\{a\}$ if $|A| = v(T^2_H)$ and there is an embedding $\phi: V(\hat{T}^2_H)\to A\cup\{a\}$ such that $\phi(u) = a.$
\end{definition}

Since $T^2_H$ is isomorphic to $H^1$, it has a perfect $H$-subdivision tiling. The graph $\hat{T}^2_H$ contains a spanning subgraph $F$ which is obtained from $T^2_H$ by replacing the edge $xy\in T^2_H$ with the path $xuy.$ Thus, $\hat{T}^2_H$ also has a perfect $H$-subdivision packing. The next definition is for type-$3$ $\mathrm{Sub}(H)$-absorbers.

\begin{definition}
    Let $H$ be a graph. Let two vertex-disjoint graphs $H^1$ and $T^1_H$ be given. Let $x, y, z, w$ be the vertices of $T^1_H$ as in \Cref{def:type1unit}. Choose an edge $x'y' \in E(H^1).$ We denote by $T^3_H$ a graph that is obtained from $H^1$ and $T^1_H$ by adding two edges $x'z$ and $y'w.$ Let $u$ be a new vertex and add edges $ux, uy$ to $T^3_H.$ We denote by $\hat{T}^3_H$ the obtained graph. We write $\tilde{T}^3_H$ to denote a bipartite graph that is $T^3_H - \{x, y\}.$

    Let $G$ be a graph and let $a\in V(G).$ We say a set $A\subseteq V(G)\setminus \{a\}$ is a \emph{type-$3$ $\mathrm{Sub}(H)$-absorber} for $\{a\}$ if $|A| = v(T^3_H)$ and there is an embedding $\phi: V(\hat{T}^3_H)\to A\cup\{a\}$ such that $\phi(u) = a.$
\end{definition}

Note that $T^3_H$ contains vertex-disjoint copies of $H^1$ and $T^1_H$ whose union covers all vertices of $T^3_H.$ We observe that $\hat{T}^3_H$ contains vertex-disjoint copies of graphs $F_1$ and $F_2$ such that $F_1$ is obtained from $H^1$ by replacing the edge $x'y'$ of $H^1$ with $x'wzy'$ and $F_2$ is obtained from $T^1_H$ by replacing $xwzy$ with $xuy.$ Both $F_1$ and $F_2$ are subdivisions of $H.$ Thus, both graphs $T^3_H$ and $\hat{T}^3_H$ have perfect $H$-subdivision tilings.

We note that $\hat{T}^2_H - u$ and $\hat{T}^3_H - u$ are both bipartite graphs. Moreover, $\hat{T}^2_H$ contains one triangle, and $\hat{T}^3_H$ contains one $C_5$, and deleting them leaves a bipartite graph. Hence, in graphs containing many $C_3$ or $C_5$, we can find many copies of $\hat{T}^2_H$ and $\hat{T}^3_H$, respectively. We will use these properties of $\hat{T}^2_H$ and $\hat{T}^3_H$ to prove a special case for \Cref{thm:hcf-2}.

Since all four graphs $T^2_H$, $\hat{T}^2_H$, $T^3_H$, and $\hat{T}^3_H$ have perfect $H$-subdivision tilings, type-$2$ and type-$3$ $\mathrm{Sub}(H)$-absorbers can be used to absorb a single vertex. Although this is an advantage over the type-$1$ $\mathrm{Sub}(H)$-absorbers, it is not guaranteed that we can always find many type-$2$ and type-$3$ absorbers in a graph $G$ if $G$ is close to being bipartite. 
Indeed, we only use type-$1$ $\mathrm{Sub}(H)$-absorbers to prove \Cref{thm:main}.
On the other hand, if $\hcf_{\xi}(H)$ is two, type-$1$ absorbers are not sufficient to build the desired absorbers. So, the proof of \Cref{thm:hcf-2} uses all three types of absorber units.

As we mentioned before, not all vertices may have many absorber units for them. So, it would be useful if we could somehow control what vertices remain uncovered. The following structure allows us to "exchange" the remaining vertices so that we can absorb those vertices better.

\begin{definition}\label{def:exchanger}
     Let $H$ be a graph and $(U, V)$ be a bipartition of $H^1$ such that $U$ consists of the branch vertices. Let $x_0y_0 \in E(H)$ be an arbitrary edge. Then there is a unique vertex $z\in V$ such that $x_0z, y_0z\in E(H^1).$ Let $S_H = H^1 - z.$ We now introduce two new vertices $u$ and $v$ and add edges $ux_0, uy_0, vx_0, vy_0$ to $S_H$ and we denote by $\hat{S}_H$ the obtained graph.

    Let $G$ be a graph and two distinct vertices $a, b\in V(G).$ We say a set $B\subseteq V(G)\setminus \{a, b\}$ is a \emph{$\mathrm{Sub}(H)$-exchanger} for $\{a, b\}$ if $|B| = v(S_H)$ and there is an embedding $\phi: V(\hat{S}_H)\to B\cup\{a, b\}$ such that $\phi(u) = a$ and $\phi(v) = b.$
\end{definition}

We assume that the choice of $ x_0y_0 $ is always the same for all $ H $, thus $ S_H $ is uniquely determined for all $ H $. We note that both $ S_H + a $ and $ S_H + b $ are subdivisions of $ H $. The following definition is a graph containing information on pairs that have many $ \mathrm{Sub}(H) $-exchangers.

\begin{definition}
    Let $ H $ be a graph and $ d > 0 $. For a given $ n $-vertex graph $ G $ and $ B \subseteq V(G) $, let $ F $ be an auxiliary graph on the vertex set $ V(G) $ such that $ uv \in E(F) $ if and only if $ u \notin B, v \in B $ and the number of $ \mathrm{Sub}(H) $-exchangers for $ \{u, v\} $ is at least $ dn^{v(S_H)} $. We say such an auxiliary graph $ F $ is a \emph{$(d, B)$-exchanging graph}.
\end{definition}

In order to prove the existence of $ \mathrm{Sub}(H) $-absorbers, we need to collect several lemmas regarding the following definition.

\begin{definition}
    Let $ n $, $ t $ be positive integers, and let $ d \in (0, 1) $. Let $ X $, $ Y $ be sets such that $ |X| \leq n^2 $ and $ |Y| = n $. We say a collection of pairs $ \calP \subseteq X \times \binom{Y}{t} $ is an \emph{$(n, t, d)$-family} on $ (X, Y) $ if $ |\{Z \subseteq Y: (x, Z) \in \calP\}| \geq dn^t $ for all $ x \in X $.
\end{definition}

The following is the key lemma to obtain $ \mathrm{Sub}(H) $-absorbers.

\begin{lemma}\label{lem:absorber-proof}
    Let $ 0 < \frac{1}{n} \ll \beta \ll d, \alpha, \frac{1}{t} \leq 1 $. Let $ \calP $ be an $(n, t, d)$-family on a pair of sets $ (X, Y) $. Then there is $ \calY \subseteq \binom{Y}{t} $ which satisfies the following:
    \begin{enumerate}
        \item[$\bullet$] $ \calY $ is a collection of pairwise disjoint sets,
        \item[$\bullet$] $ |\calY| \leq \frac{\alpha}{t} n $ and $ |\calY| $ is even,
        \item[$\bullet$] for every $ x \in X $, the size of the set $ \{Z \in \calY: (x, Z) \in \calP\} $ is at least $ \beta n $.
    \end{enumerate}
\end{lemma}

\begin{proof}
    Let $ c = \min\{\frac{\alpha}{2}, \frac{d}{10}\} $. Consider a random family $ \calT $ of size $ t $ subsets of $ Y $, such that each set in $ \binom{Y}{t} $ is selected independently at random with the probability $ p = \frac{c}{n^{t - 1}} $. Then $ |\calT| $ has a binomial distribution with an expectation less than or equal to $ \frac{\alpha}{2 t} n $. By Chebyshev's inequality, with probability $ 1 - o(1) $, the following holds.
    \begin{equation}\label{eq:1}
        |\calT| \leq \frac{\alpha}{t} n.
    \end{equation}

    For each $ x \in X $, let $ \calZ_x \defeq \{Z \subseteq Y: (x, Z) \in \calP\} $. Then the size $ |\calT \cap \calZ_x| $ has a binomial distribution with an expectation $ |\calZ_x| p \geq cd n $ for every $ x \in X $. By Chernoff bound, for each $ x \in X $, the inequality $ |\calT \cap \calZ_x| \geq \frac{cd}{2} n $ holds with probability $ 1 - o(n^{-3}) $. Since $ |X| \leq n^2 $, we have the following with probability $ 1 - o(1) $.
    \begin{equation}\label{eq:2}
        |\calT \cap \calZ_x| \geq \frac{cd}{2} n \text{ for all } x \in Z.
    \end{equation}
    
    We now count the number of intersecting pairs in $ \calT $. Let $ I $ be a random variable which is the number of intersecting pairs of $ \calT $. Since we have $ |\{(P, Q): P, Q \subseteq [n],\text{ } P \cap Q \neq \emptyset, \text{ and } |P| = |Q| = t\}| \leq n^{2t - 1} $, and each $ P $, $ Q $ are chosen to be in $ \calT $ with probability $ p = \frac{c}{n^{t-1}} $, we have $ \mathbb{E}(I) \leq c^2 n $. By Markov's inequality, with probability at least $ \frac{1}{2} $,
    \begin{equation}\label{eq:3}
        |I| \leq 2c^2 n.
    \end{equation}

    Thus, for all sufficiently large $ n $, there is a family $ \calT $ that satisfies all \Cref{eq:1,eq:2,eq:3}. Let $ \calY $ be a collection obtained from $ \calT $ by removing intersecting pairs from $ \calT $ and removing at most one arbitrary element to ensure that $ |\calY| $ is even.

    Then $ \calY $ is a collection of pairwise disjoint sets and by \eqref{eq:1}, we have $ |\calY| \leq |\calT| \leq \frac{\alpha}{t} n $ and $ |\calY| $ is even. By \eqref{eq:2} and \eqref{eq:3}, the size of the intersection $ |\calY \cap \calZ_x| \geq \frac{cd}{2} n - 2c^2 n - 1 \geq \frac{\alpha d^2}{100} n $ for all $ x \in X $. We now set $ \beta = \frac{\alpha d^2}{100} $. Then $ \calY $ is the desired collection. This completes the proof.
\end{proof}

We remark that \Cref{lem:absorber-proof} is a systematization of Lemma 2.3 from~\cite{rodl2006dirac}. One direct application of \Cref{lem:absorber-proof} is the existence of one-side perfect matchings for small sets, which is the following:

\begin{lemma}\label{lem:system-matching}
    Let $0 < \frac{1}{n} \ll \beta \ll \alpha, \eta < 1.$ 
    Let $G$ be an $n$-vertex graph. Suppose that a vertex subset $B \subseteq V(G)$ has size at least $\eta n$ and all vertices $v \in V(G) \setminus B$ satisfy $d(v, B) \geq \eta n.$ Then there is a set $B' \subseteq B$ such that $|B'| \leq \alpha n$ and for any $U \subseteq V(G) \setminus B$ with size at most $\beta n$, there is a matching on $G[U, B']$ which covers all vertices in $U.$   
\end{lemma}

\begin{proof}
    Let $X = V(G) \setminus B.$ Let $\calP = \{(x, \{y\}) \in X \times \binom{V(G)}{1} : xy \in E(G[X, B])\}.$ Then $\calP$ is an $(n, 1, \eta)$-family on $(X, V(G)).$ As $\frac{1}{n} \ll \beta \ll \alpha, \eta$, by applying \Cref{lem:absorber-proof} with the parameters $\alpha$, $\eta$, $\beta$, and $n$ playing the roles of $\alpha$, $d$, $\beta$ and $n$, respectively, we obtain $\calY \subseteq \binom{B}{1}$ satisfying the following. $|\calY| \leq \alpha n$ and $|\{(v, \{b\}) \in \calP : \{b\} \in \calY\}| \geq \beta n$ for all $x \in X.$
    Let $B' = \bigcup_{\{b\} \in \calY} \{b\}.$ Then $|B'| \leq \alpha n$ and for any vertex $v \in V(G) \setminus B$, the degree $d(v, B') \geq \beta n.$ Thus, for any subset $U \subseteq V(G) \setminus B$ with size at most $\beta n$, we can find a matching $M$ greedily in $G[U, B']$ which covers all vertices of $U.$
\end{proof}

Using \Cref{lem:absorber-proof}, we can obtain the following lemma.

\begin{lemma}\label{lem:exchanging-graph}
    Let $H$ be an $h$-vertex graph and $0 < \frac{1}{n}, \beta \ll d, \alpha, \eta, \frac{1}{h} < 1.$
    Let $G$ be an $n$-vertex graph and a vertex subset $B \subseteq V(G)$ has size at least $\eta n.$ Let $F$ be the $(d, B)$-exchanging graph of $G.$ Let $U = \{v \in V(G) \setminus B : d_F(v) \geq \frac{\eta}{3} n\}.$ Then there exists a set $A \subseteq V(G)$ satisfying the following:
    \begin{enumerate}
        \item[$\bullet$] $G[A]$ has a perfect $H$-subdivision tiling,
        \item[$\bullet$] $|A| \leq \alpha n$,
        \item[$\bullet$] for any $U' \subseteq U \setminus A$ with $|U'| \leq \beta n$, there exists a subset $W \subseteq A \cap B$ such that $|W| = |U'|$ and $G[(A \cup U') \setminus W]$ has a perfect $H$-subdivision tiling.
    \end{enumerate}
\end{lemma}

\begin{proof}
    We recall that $\hat{S}_H$ is a graph defined in \Cref{def:exchanger} which has two vertices $u$ and $v$ such that both $S_H - u$ and $S_H - v$ contain spanning subdivisions of $H.$ 

    Let $$\calP = \{(u, Z) \in U \times \binom{V(G)}{v(S_H) + 1} : \exists v \in Z \cap B \text{ s.t. } Z - v \text{ is a } \mathrm{Sub}(H) \text{-exchanger for } \{u, v\}\}.$$
    For each $u \in U$, the number of pairs $(v, S) \in B \times \binom{V(G)}{v(S_H)}$ such that $S$ is a $\mathrm{Sub}(H)$-exchanger for $\{u, v\}$ is at least $\frac{\eta d}{3} n^{v(S_H) + 1}.$ Thus, $\calP$ is an $(n, v(S_H) + 1, \frac{\eta d}{3(v(S_H) + 1)})$-family on $(U, V(G)).$ 

    As $v(S_H) + 1 \leq h^2$ and $\frac{1}{n}, \beta \ll d, \alpha, \eta, \frac{1}{h}$, by applying \Cref{lem:absorber-proof} with $v(S_H) + 1$, $\alpha$, $\frac{\eta d}{3(v(S_H) + 1)}$ and $\beta (v(S_H) + 1)$ playing the roles of $t$, $\alpha$, $d$ and $\beta$, respectively, we obtain a collection of vertex-disjoint subsets $\calY \subseteq \binom{V(G)}{v(S_H) + 1}$ satisfying the following.

    \begin{enumerate}
        \item $|\calY| \leq \frac{\alpha}{v(S_H) + 1} n$,
        \item $|\{Z \in \calY : (u, Z) \in \calP \}| \geq \beta (v(S_H) + 1) n$ for all $u \in U.$
    \end{enumerate}
    
    Let $A \defeq \bigcup_{Z \in \calY} Z.$ Then $|A| \leq \alpha n.$
    Let $U' \subseteq U \setminus A$ with size at most $\beta n.$ 
    By our choice of $U$ and $\calY$, for each vertex $x \in U'$, there are at least $\frac{\beta (v(S_H) + 1)}{v(S_H) + 1} n = \beta n$ distinct pairs $(u, v)$ which use distinct $\mathrm{Sub}(H)$-exchangers in $\calY$ such that $v \in A \cap B$ and there is $Z \in \calY$, where $v \in Z$ and $Z - v$ is a $\mathrm{Sub}(H)$-exchanger for $\{u, v\}.$
    Thus, we can greedily exchange all vertices in $U'$ with a set of distinct vertices $W$, by using $\mathrm{Sub}(H)$-exchangers, where $W \subseteq A \cap B$ and $|W| = |U'|.$
\end{proof}

\subsection{Type-$1$ $\mathrm{Sub}(H)$-absorbers}
The following lemma is useful to get $\mathrm{Sub}(H)$-absorbers which use type-1 $\mathrm{Sub}(H)$-absorbers.

\begin{lemma}\label{lem:type1-absorber}
    Let $H$ be an $h$-vertex graph and $0 < \frac{1}{n}, \beta \ll d, \alpha, \frac{1}{h} \leq 1.$
    Let $G$ be an $n$-vertex graph and $F$ be the $d$-absorbing graph of $G.$ Then there exists a set $A \subseteq V(G)$ satisfying the following.
    \begin{enumerate}
        \item[$\bullet$] $ |A| \leq \alpha n $,
        \item[$\bullet$] $ |A| $ is even,
        \item[$\bullet$] The set $A$ is a $\mathrm{Sub}(H)$-absorber for any vertex set $U$ such that $U \subseteq V(G) \setminus A$, $ |U| \leq \beta n $, and $F[U]$ has a perfect matching.
    \end{enumerate}
\end{lemma}

\begin{proof}
    Let 
    $$
    \calP = \{(e, Z) \in E(F) \times \binom{V(G)}{v(T^1_H)} : Z \text{ is a type-1 } \mathrm{Sub}(H)\text{-absorber for } e\}.
    $$ 
    Then $\calP$ is an $(n, v(T^1_H), d)$-family on $(E(F), V(G))$. By applying \Cref{lem:absorber-proof} with $v(T^1_H)$, $\alpha$, $d$, and $\beta$ playing the roles of $t$, $\alpha$, $d$, and $\beta$, respectively, we obtain a collection of vertex-disjoint subsets $\calY \subseteq \binom{V(G)}{v(T^1_H)}$ satisfying the following.

    \begin{enumerate}
        \item $ |\calY| \leq \frac{\alpha}{v(T^1_H)} n $,
        \item $ |\calY| $ is even,
        \item $ |\{Z \in \calY : (e, Z) \in \calP \}| \geq \beta n $ for all $ e \in E(F) $.
    \end{enumerate}
    
    Let $A = \bigcup_{Z \in \calY} Z$. Then $ |A| $ is even and less than or equal to $\alpha n$. 
    Let $U$ be a subset of $V(G) \setminus A$ with size at most $\beta n$. Assume $F[U]$ has a perfect matching $M$. Then $ e(M) \leq \frac{\beta}{2} n $. By our choice of $\calY$, for each edge $e \in E(M)$, there are at least $\beta n$ distinct type-1 $\mathrm{Sub}(H)$-absorbers. Thus, $A$ can absorb every pair $e \in E(M)$ greedily. This means $A$ is a $\mathrm{Sub}(H)$-absorber for $U$. This proves the lemma. 
\end{proof}

\subsection{Type-$2$ and type-$3$ $\mathrm{Sub}(H)$-absorbers}
In this subsection, we will show how we can obtain $\mathrm{Sub}(H)$-absorbers via type-2 and type-3 $\mathrm{Sub}(H)$-absorbers, respectively.

\begin{lemma}\label{lem:type2-absorber}
    Let $H$ be an $h$-vertex graph and $0 < \frac{1}{n}, \beta \ll d, \alpha, \eta, \frac{1}{h} \leq 1.$ Let $j \in \{2, 3\}.$
    Let $G$ be an $n$-vertex graph and let $B \subseteq V(G)$ with size at least $\eta n$ such that for all $v \in B$, there are at least $dn^{v(T^j_H)}$ type-$j$ $\mathrm{Sub}(H)$-absorbers for $v.$ Let $F_1$ and $F_2$ be the $d$-absorbing graph and the $(d, B)$-exchanging graph of $G$, respectively. Assume that for every $v \in V(G) \setminus B$, at least one of the inequalities $d_{F_1}(v; B) \geq \frac{\eta}{3} n$ or $d_{F_2}(v; B) \geq \frac{\eta}{3} n$ holds. Then, there is a set $A \subseteq V(G)$ with size at most $\alpha n$ such that for all $U \subseteq V(G) \setminus A$ with size at most $\beta n$, the set $A$ is a $\mathrm{Sub}(H)$-absorber for $U.$
\end{lemma}

\begin{proof}
    Choose additional constants $\beta_0, \alpha_0, \alpha_1$ so that the following holds:
    $$
    0 < \frac{1}{n}, \beta \ll \beta_0 \ll \alpha_0 \ll \alpha_1 \ll d, \alpha, \eta, \frac{1}{h} \leq 1.
    $$
    We write $R = V(G) \setminus B$, and $X = \{v \in R : d_{F_1}(v; B) \geq \frac{\eta}{3} n\}$. Let $Y = R \setminus X.$ We note that for all $v \in Y$, we have $d_{F_2}(v; B) \geq \frac{\eta}{3} n.$ We now collect the following two claims.

    \begin{claim}\label{clm:type2-blue-absorber}
        There exists a set $A_0 \subseteq V(G)$ satisfying the following. The size of the set $A_0$ is less than or equal to $\alpha_0 n$, and for any $U \subseteq B \setminus A_0$ with size at most $\beta_0 n$, the set $A_0$ is a $\mathrm{Sub}(H)$-absorber for $U.$
    \end{claim}

    \begin{claimproof}[Proof of \Cref{clm:type2-blue-absorber}]
        Let 
        $$
        \calP = \{(x, Z) \in B \times \binom{V(G)}{v(T^j_H)} : Z \text{ is a type-}j \text{ } \mathrm{Sub}(H)\text{-absorber for } x\}.
        $$
        Then $\calP$ is an $(n, v(T^j_H), d)$-family on $(B, V(G))$. By applying \Cref{lem:absorber-proof} with $v(T^j_H)$, $\alpha_0$, $d$, and $\beta_0$ playing the roles of $t$, $\alpha$, $d$, and $\beta$, respectively, we obtain a collection of pairwise disjoint sets $\calY \subseteq \binom{V(G)}{v(T^j_H)}$ satisfying the following.
        \begin{enumerate}
            \item $ |\calY| \leq \frac{\alpha_0}{v(T^j_H)} n $,
            \item $ |\{Z \in \calY : (x, Z) \in \calP \}| \geq \beta_0 n $ for all $x \in B.$
        \end{enumerate}
        Let $A_0 \defeq \bigcup_{Z \in \calY} Z.$ Then $ |A_0| \leq \alpha_0 n.$
        Let $U \subseteq B \setminus A_0$ with size at most $\beta_0 n.$ By our choice of $\calY$, for each vertex $x \in U$, there are at least $\beta_0 n$ distinct type-$j$ $\mathrm{Sub}(H)$-absorbers in $A_0$. Thus, $A_0$ can absorb all vertices in $U$ greedily. This means $A_0$ is a $\mathrm{Sub}(H)$-absorber for $U.$ This proves the claim.
    \end{claimproof}

    \begin{claim}\label{clm:type2-type1}
There exists a set $ A_1 \subseteq V(G) $ with size at most $ \alpha_1 n $ satisfying the following: for any $ U \subseteq (B \cup X) \setminus A_1 $ with size at most $ \beta n $, the set $ A_1 $ is a $\mathrm{Sub}(H)$-absorber for $ U $.
\end{claim}

\begin{claimproof}[Proof of \Cref{clm:type2-type1}]
We fix additional parameters $\beta''_0, \alpha''_0, \beta'_0, \alpha'_0$ so that the following holds: $\beta \ll \beta''_0 \ll \alpha''_0 \ll \beta_0, \beta'_0 \ll \alpha_0, \alpha'_0 \ll \alpha_1$. By \Cref{clm:type2-blue-absorber}, there exists a set $ A_0 \subseteq B $ such that $ |A_0| \leq \alpha_0 n $ and for any $ U \subseteq B \setminus A_0 $ with size at most $ \beta_0 n $, the set $ A_0 $ is a $\mathrm{Sub}(H)$-absorber for $ U $. Let $ G_1 = G - A_0 $, $ B_1 = B \setminus A_0 $, and $ X_1 = X \setminus A_0 $. Then $ v(G_1) \geq (1-\alpha_0)n $ and $ |B_1| \geq \frac{\eta}{2} n $. By \Cref{obs:delet-small-vertices}, we know that for every vertex $ v \in B_1 $, the number of type-$ j $ $\mathrm{Sub}(H)$-absorbers in $ G_1 $ is at least $\frac{d}{2}n_1^{v(T^j_H)}$, and for each pair of vertices $ (u, v) \in X_1 \times B_1 $, where $ uv \in E(F_1) $, the number of type-1 $\mathrm{Sub}(H)$-absorbers for $ \{u, v\} $ is at least $\frac{d}{2}n_1^{v(T^1_H)}$. Let $ F'_1 $ be the $\frac{d}{2}$-absorbing graph of $ G_1 $. Then for every vertex $ v \in X_1 $, the degree $ d_{F'_1}(v, B_1) \geq \frac{\eta}{3}n - \alpha_0 n \geq \frac{\eta}{4}n $.

By \Cref{lem:type1-absorber}, there is $ A'_0 \subseteq V(G_1) $ with size at most $ \alpha'_0 n $ such that for any $ U \subseteq V(G_1) \setminus A'_0 $ with size less than $ \beta'_0 n $ which has a perfect matching in $ F'_1 $, the set $ A'_0 $ is a $\mathrm{Sub}(H)$-absorber for $ U $.

Let $ G_2 = G_1 - A'_0 $, $ B_2 = B_1 \setminus A'_0 $, and $ X_2 = X_1 \setminus A'_0 $. Then $ v(G_2) \geq (1 - \alpha_0 - \alpha'_0)n $ and $ |B_2| \geq \frac{\eta}{4} n $. For every $ v \in X_2 $, the degree $ d_{F'_1}(v;B_2) \geq \frac{\eta}{4} n - \alpha'_0 n \geq \frac{\eta}{10} n $. Then by \Cref{lem:system-matching}, there is a set $ A''_0 \subseteq B_2 $ with size at most $ \alpha''_0 n $ such that for any $ U \subseteq X_2 \setminus A''_0 $ with size less than $ \beta''_0 n $, the graph $ F'_1[U, A''_0] $ has a matching $ M $ which covers all vertices of $ U $.

Now, let $ A_1 = A_0 \cup A'_0 \cup A''_0 $. Then the size $ |A_1| \leq (\alpha_0 + \alpha'_0 + \alpha''_0)n \leq \alpha_1 n $. Note that since $ A''_0 \subseteq B_1 $ and $ |A''_0| \leq \alpha''_0 n \leq \frac{\beta'_0}{2} n $, the induced graph $ G[A_0 \cup A''_0] $ has a perfect $ H $-subdivision tiling. Since $ G[A'_0] $ also has a perfect $ H $-subdivision tiling and $ A_0 $, $ A'_0 $, $ A''_0 $ are vertex-disjoint subsets, $ G[A_1] $ has a perfect $ H $-subdivision tiling.

We now claim that $ A_1 $ is the desired $\mathrm{Sub}(H)$-absorber. Let $ U \subseteq (B \cup X) \setminus A_1 $ with size at most $ \beta n $. Let $ U_B = U \cap B $ and $ U_X = U \cap X $. Since $ |U_X| \leq \beta n \leq \beta''_0 n $, by our choice of $ A''_0 $, there is a set $ W \subseteq A''_0 $ with the same size as $ U_X $ such that $ F'_1[U_X, W] $ has a perfect matching. Since $ |U_X \cup W| \leq 2\beta n \leq \beta'_0 n $, the induced graph $ G[(A'_0 \cup U_X \cup W)] $ has a perfect $ H $-subdivision tiling. Let $ W' = A''_0 \setminus W $. We note that $ |W' \cup U_B| \leq \beta n \leq \beta_0 n $ and $ (W' \cup U_B) \subseteq B_1 $, so by our choice of $ A_1 $, the induced graph $ G[(A_0 \cup W' \cup U_B)] $ has a perfect $ H $-subdivision tiling. Thus, $ G[A_1 \cup U] $ also has a perfect $ H $-subdivision tiling. Therefore, the set $ A_1 $ is a $\mathrm{Sub}(H)$-absorber for $ U $.
\end{claimproof}
    
By \Cref{clm:type2-type1}, there exists a set $A_1 \subseteq V(G)$ with size at most $\alpha_1 n$ such that for any subset $U$ of $(B \cup X) \setminus A_1$ with size at most $\beta n$, the set $A_1$ is a $\mathrm{Sub}(H)$-absorber for $U$. 
Let $G_1 = G - A_1$, $B_1 = B \setminus A_1$, $X_1 = X \setminus A_1$, and $Y_1 = Y \setminus A_1$. Then $v(G_1) \geq (1 - \alpha_1)n$ and $|B_1| \geq \frac{\eta}{2} n$. By \Cref{obs:delet-small-vertices}, we know that for each pair of vertices $(u, v) \in Y_1 \times B_1$, where $uv \in E(F_2)$, the number of $\mathrm{Sub}(H)$-exchangers for $\{u, v\}$ is at least $\frac{d}{2} n_1^{v(S_H)}$.
Let $F'_2$ be the $(\frac{d}{2}, B_1)$-exchanging graph of $G_1$. Then for every vertex $v \in Y_1$, the degree $d_{F'_2}(v, B_1) \geq \frac{\eta}{3}n - \alpha_1 n \geq \frac{\eta}{4} n$.

By \Cref{lem:exchanging-graph}, there is a subset $A_2$ of $V(G_1)$ with size at most $\alpha_1 n$ such that $G_1[A_2]$ has a perfect $H$-subdivision tiling and satisfies the following: for any $U \subseteq Y_1 \setminus A_2$ with size at most $\beta n$, there is a set $W \subseteq (A_2 \cap X_1)$ with the same size as $U$ such that $G_1[(A_2 \cup U) \setminus W]$ also has a perfect $H$-subdivision tiling.

Now, let $A = A_1 \cup A_2$. Then the inequality $|A| \leq 2 \alpha_1 n \leq \alpha n$ holds. Note that both $G[A_1]$ and $G[A_2]$ have perfect $H$-subdivision tilings and $A_1$ and $A_2$ are disjoint subsets of $V(G)$. Thus, $G[A]$ also has a perfect $H$-subdivision tiling.

We now claim that $A$ is the desired $\mathrm{Sub}(H)$-absorber. Let $U \subseteq V(G) \setminus A$ with size at most $\beta n$. Let $U_B = U \cap B$, $U_X = U \cap X$, and $U_Y = U \cap Y$. Since $|U_Y| \leq \beta n$, by our choice of $A_2$, there is a set $W \subseteq (A_2 \cap X_1)$ with the same size as $U_Y$ such that $G[(A_2 \cup U_Y) \setminus W]$ has a perfect $H$-subdivision tiling. Note that $|U_B \cup U_X \cup W| = |U| \leq \beta n$ and $(U_B \cup U_X \cup W) \subseteq (B \cup X) \setminus A_1$. Thus, $U_B \cup U_X \cup W$ can be absorbed by $A_1$. 
Thus, $G[A \cup U]$ has a perfect $H$-subdivision tiling. Therefore, $A$ is a $\mathrm{Sub}(H)$-absorber for $U$. This completes the proof.
\end{proof}

\section{Proofs}\label{sec:proof}


\subsection{Proof sketch}\label{subsec:sketch}

In this section, we will prove our main results, \Cref{thm:main,thm:hcf-2}. By \Cref{prop:not-exceed-1/2}, \Cref{prop:lowerbound-hcf}, \Cref{prop:hcf-not-2}, and \Cref{prop:hcf-2-1/3}, it suffices to prove the following two lemmas.

\begin{lemma}\label{lem:lemma-for-main1}
    Let $H$ be an $h$-vertex graph with $\hcf_{\xi}(H) = 1$. Let $0 < \frac{1}{n} \ll \gamma, \frac{1}{h} \leq 1$. Let $G$ be an $n$-vertex graph. Then the following holds: If $\delta(G) \geq \left(1 - \frac{1}{\xi(H)} + \gamma \right)n$, then $G$ has a perfect $H$-subdivision tiling.
\end{lemma}

\begin{lemma}\label{lem:lemma-for-main2}
    Let $H$ be an $h$-vertex graph with $\hcf_{\xi}(H) = 2$. Let $0 < \frac{1}{n} \ll \gamma, \frac{1}{h} \leq 1$. Let $G$ be an $n$-vertex graph, where $n$ is even. Then the following holds: If $\delta(G) \geq \left(\max\{\frac{1}{3}, 1 - \frac{1}{\xi(H)} \} + \gamma \right)n$, then $G$ has a perfect $H$-subdivision tiling.
\end{lemma}

The proofs of \Cref{lem:lemma-for-main1,lem:lemma-for-main2} will be provided in \Cref{subsec:proof-main1,subsec:proof-main2}, respectively. The proofs crucially use the absorption method. Our proof strategy is as follows: Let $H$ be a graph and $G$ be a host graph equipped with a minimum degree condition as in the statements of the above lemmas, in which we want to find a perfect $H$-subdivision tiling. Then, together with \Cref{obs:xi-chicr} and \Cref{thm:chritical-constant-missing}, we can find an almost perfect $H$-subdivision tiling in $G$. Thus, we mainly focus on absorbing the remaining vertices by constructing suitable absorbers for all cases. A detailed proof sketch is provided below.

\begin{description}
    \item[Step 1: Preprocessing.] Apply the regularity lemma with properly chosen parameters on $G$. After that, collect bad vertices and construct absorbing graphs on good vertices.
    
    \item[Step 2: Place the absorber.] Place an efficient $\mathrm{Sub}(H)$-absorber by using the abundant properties of $\varepsilon$-regular pairs.
    
    \item[Step 3: Cover almost all vertices.] Find a set $W \subseteq V(G)$ outside of the $\mathrm{Sub}(H)$-absorber obtained from Step 2 such that $G[W]$ has a perfect $H$-subdivision tiling, $W$ contains all bad vertices, and $W$ covers all but at most a constant number of vertices of $G$ that are not in the $\mathrm{Sub}(H)$-absorber.
    
    \item[Step 4: Absorb the uncovered vertices.] Absorb all the vertices that are not covered in Step 3 by using the $\mathrm{Sub}(H)$-absorber obtained from Step 2.
\end{description}

In the above description, we omitted many details. To achieve Step 3, the following lemma would be useful.

\begin{lemma}\label{lem:cover-almost}
    Let $H$ be an $h$-vertex graph. Let $0 < \frac{1}{n} \ll \alpha, \frac{1}{h} \leq 1$, and $0 < \frac{1}{C} \ll \frac{1}{h} \leq 1$. Let $G$ be an $n$-vertex graph with minimum degree at least $\left(1 - \frac{1}{\xi(H)} + \alpha \right)n$. Let $X$ be a subset of $V(G)$ with size at most $\frac{\alpha}{2v(H^*)}n$. Then there is a set $W \subseteq V(G)$ with size at least $n - C$ such that $X \subseteq W$ and $G[W]$ have a perfect $H$-subdivision tiling. Moreover, if $\hcf_{\xi}(H) = 2$, we can get such a set $W$ whose size is even.
\end{lemma}

\begin{proof}
    We will first find vertex-disjoint copies of $ H^* $ covering all vertices of $ X $. Let $ X = \{v_1, \dots, v_t\} $ where $ t \leq \frac{\alpha}{2v(H^*)}n $. Assume for $ i < t $, we have $ \calF_i = \{F_1, \dots, F_i\} $, where $ \calF_i $ is a collection of vertex-disjoint copies of $ H^* $ and $ v_j \in V(F_j) $ for each $ j \in [i] $, and $ G_i = G - \bigcup_{j \in [i]} V(F_i) - X $.

    We note that $ \left| \bigcup_{j \in [i]} V(F_i) \cup X \right| \leq \frac{\alpha}{2}n $ holds as $ i < t $. Thus, $ \delta(G_i + v_{i+1}) \geq \left( 1 - \frac{1}{\xi(H^*)} + \frac{\alpha}{2} \right)n $. Let $ Y $ be the subset of the neighborhoods of $ v_{i+1} $ in $ G_i + v_{i+1} $ of size $ \left( 1 - \frac{1}{\xi(H^*)} + \frac{\alpha}{2} \right)n $. Let $ A $ and $ B $ be the bipartition of $ H^* $ and $ u $ be a vertex in $ A $. We will embed $ u $ to $ v_{i+1} $ first and then embed other vertices of $ H^* - u $ while all the vertices of $ B $ are embedded in $ Y $. By the minimum degree condition on $ G_{i+1} $, \Cref{lem:fix-embedding} enables us to find such embedding, we obtain a new copy of $ H^* $, say $ F_{i+1} $ containing $ v_{i+1} $. As it is disjoint with $ \bigcup_{j \leq i} V(F_j) \cup X $, we may add it to enlarge $ \calF_i $ to $ \calF_{i+1} $. We iterate this process until we get $ \calF_t $. Let $ F = \bigcup_{i \in \calF_t} V(F_i) $. Then $ X \subseteq F $ and $ |F| = v(H^*) t \leq \frac{\alpha}{2}n $ holds.

    Now, let $ G' = G - F $. Since $ |F| \leq \frac{\alpha}{2}n $, the minimum degree of $ G' $ is at least $ \left( 1 - \frac{1}{\xi(H)} + \frac{\alpha}{2} \right)n $. Thus, by \Cref{thm:chritical-constant-missing} and \Cref{obs:xi-chicr}, there is a set $ F' \subseteq V(G') $ such that $ G'[F'] $ has a perfect $ H^* $-tiling and $ v(G') - |F'| \leq C $. Let $ W = F \cup F' $. Then $ W $ is the desired set.

    When $ \hcf_{\xi}(H) = 2 $, since the order of $ v(H^*) $ is even, the size of $ W $ is also even. This completes the proof.
\end{proof}

In the proof of \Cref{lem:lemma-for-main1}, we will use the concept of dominating sets to obtain an efficient $\mathrm{Sub}(H)$-absorber. Let $G$ be a graph. We say $D$ is a \emph{dominating set} of $G$ if every vertex $v \in V(G)$ is either contained in $D$ or has a neighbor in $D$. The domination number of $G$ is defined as the minimum size of a dominating set of $G$.

The next lemma on domination number, proven by Arnautov~\cite{arnautov1974estimation}, and independently by Payan~\cite{payan1975nombre}, will be used in the proof of \Cref{lem:lemma-for-main1}.

\begin{lemma}[Arnautov~\cite{arnautov1974estimation}, Payan~\cite{payan1975nombre}]\label{lem:bounded-domination}
    Let $G$ be an $n$-vertex graph with minimum degree $\delta$. Then the domination number of $G$ is bounded above by $\frac{1 + \ln (\delta + 1)}{\delta + 1}n$.
\end{lemma}

Later, Alon~\cite{alon1990transversal} proved that \Cref{lem:bounded-domination} is asymptotically best possible.

In order to use type-$1$ $\mathrm{Sub}(H)$-absorbers, we need to find a perfect matching in a $d$-absorbing graph for a suitable constant $d$. To achieve this, it would be useful to define a graph as follows.

\begin{definition}\label{def:hatH}
    For a given graph $H$, consider a family of disjoint unions of bipartite subdivisions of $H$. Among those, choose a graph $\hat{H}$ with bipartition $(\hat{A}, \hat{B})$ with the smallest possible $v(\hat{H})$ such that $||\hat{A}| - |\hat{B}|| = \hcf_{\xi}(H)$. If there are multiple choices of $\hat{H}$, $\hat{A}$, and $\hat{B}$, we fix one choice.  
\end{definition}

We note that for every $H$, we can prove that the bipartite graph $\hat{H}$ exists in the following way. By the definition of $\hcf_{\xi}(H)$, there exist bipartite graphs $H_1, \dots, H_m$ with bipartitions $(A_1, B_1), \dots, (A_m, B_m)$, respectively, that are bipartite subdivisions of $H$ such that they satisfy the following: there are positive integers $c_1, \dots, c_m$ that satisfy the inequality $\sum_{i \in [m]} c_i(|A_i| - |B_i|) = \hcf_{\xi}(H)$. By taking a disjoint union of $c_i$ copies of $H_i$ for each $i \in [m]$, we obtain a bipartite graph that is a disjoint union of bipartite subdivisions of $H$ and the difference between the bipartitions is $\hcf_{\xi}(H)$. Among such bipartite graphs, we can take our desired graph $\hat{H}$.

We use $\hat{H}$ and the bipartition $(\hat{A}, \hat{B})$ to find perfect matchings in a $d$-absorbing graph for using type-$1$ $\mathrm{Sub}(H)$-absorbers. The following observation shows that complete bipartite graphs with appropriate sizes act like reservoirs for $\hat{H}$.

\begin{observation}\label{obs:hat-H-reservoir}
    Let $H$ be a graph with $\hcf_{\xi}(H) = t$ and assume $\hat{H}$ has a bipartition with sizes $h'$ and $h' + t$. Let $b$ be a positive integer. Then for any non-negative integer $a \leq b$, the complete bipartite graph with bipartition sizes $(2h' + t)b - h'a$ and $(2h' + t)b - (h' + t)a$ has a perfect $\hat{H}$-tiling.
\end{observation}

\begin{proof}
    Let $x = (2h' + t)b - h'a$ and $y = (2h' + t)b - (h' + t)a$.
    Since $x = h' (b - a) + (h' + t)b$ and $y = (h' + t)(b - a) + h' b$, the complete bipartite graph with bipartition sizes $x$ and $y$ has $2b - a$ vertex-disjoint copies of $K_{h', h' + t}$. Thus, $K_{x, y}$ has a perfect $\hat{H}$-tiling.
\end{proof}

\subsection{Proof of \Cref{thm:main}}\label{subsec:proof-main1}
In this subsection, we will prove \Cref{lem:lemma-for-main1}, which will complete the proof of \Cref{thm:main}.

\begin{proof}[Proof of \Cref{lem:lemma-for-main1}]
Let $H$ be an $h$-vertex graph with $\hcf_{\xi}(H) = 1$. Let $\hat{H}$ be the bipartite graph as in \Cref{def:hatH} with a bipartition $(\hat{A}, \hat{B})$ with sizes $h'$ and $h'+1$. Note that $h'$ is bounded above by a constant that depends only on $h$. We fix positive constants as follows:
$$0 < \frac{1}{n} \ll \beta \ll \alpha \ll d_2 \ll \frac{1}{T} \ll \frac{1}{t_0} \ll \varepsilon \ll d_1 \ll d \ll \frac{1}{h'}, \frac{1}{C} \ll \gamma, \frac{1}{h} \leq 1.$$
Let $n \geq n_0$ and $G$ be an $n$-vertex graph with minimum degree at least $\left(1 - \frac{1}{\xi(H)} + \gamma \right)n$. We apply the regularity lemma to $G$ with parameters $\varepsilon$, $d$, and $t_0$ playing the roles of $\varepsilon$, $d$, and $t_0$, respectively. Let $V_0, V_1, \dots, V_t$ be an $\varepsilon$-regular partition of $G$ and $G_0$ be a spanning subgraph of $G$ that satisfies the following conditions:

\begin{itemize}
    \item[(R1)] $t_0 \leq t \leq T_0$,
    \item[(R2)] $|V_0| \leq \varepsilon v(G)$,
    \item[(R3)] $|V_1| = \cdots = |V_t|$,
    \item[(R4)] $G_0[V_i]$ has no edges for each $i \in [t]$,
    \item[(R5)] $d_{G_0}(v) \geq d_G(v) - (d + \varepsilon)v(G)$ for all $v \in V(G)$,
    \item[(R6)] for all $1 \leq i < j \leq t$, either the pair $(V_i, V_j)$ is an $(\varepsilon, d^+)$-regular pair or $G_0[V_i, V_j]$ has no edges.
\end{itemize}

By (R5), the minimum degree of $G_0$ is at least $\left(1 - \frac{1}{\xi(H)} + \frac{\gamma}{2} \right)n$. Let $R$ be the $(\varepsilon, d)$-reduced graph on $\{V_1, \dots, V_t\}$. Then, by \Cref{lem:reduced-minimum-degree}, $\delta(R) \geq \left(1 - \frac{1}{\xi(H)} + \frac{\gamma}{2} \right)t$. By \Cref{lem:bounded-domination}, there is a dominating set $D \subseteq V(R)$ such that
$$|D| \leq \frac{1 + \ln \left(\left(1 - \frac{1}{\xi(H)} + \frac{\gamma}{2} \right)t + 1\right)}{\left(1 - \frac{1}{\xi(H)} + \frac{\gamma}{2} \right)t + 1}t \leq \frac{\varepsilon}{2}t.$$
We can pick a set $D' \subseteq V(R) \setminus D$ greedily such that $R[D, D']$ has a perfect matching $M$. Let
$$D = \{X_1, \dots, X_{\ell}\}, \quad D' = \{Y_1, \dots, Y_{\ell}\}, \quad \text{and} \quad E(M) = \{X_1Y_1, \dots, X_{\ell}Y_{\ell}\}.$$
We note that $v(M) = 2\ell \leq \varepsilon t$. Since $D$ is a dominating set of $R$, there is a partition $\{\mathcal{V}_1, \dots, \mathcal{V}_{\ell}\}$ for $V(R) \setminus (D \cup D')$ such that for every $V_{i,j} \in \mathcal{V}_i$, the edge $X_i V_{i,j} \in E(R)$. Let $P_i \defeq \bigcup_{V_{i,j} \in \mathcal{V}_i} V_{i,j} \cup Y_i$ for each $i \in [\ell]$.

For each $V_{i,j} \in \mathcal{V}_i$, where $i \in [\ell]$, let $U_{i,j} = \{u \in V_{i,j} : d_{G_0}(u; X_i) < (d - \varepsilon)|X_i|\}$. By \Cref{lem:super-regular-delete}, the size of $U_{i,j}$ is at most $\varepsilon |V_{i,j}| \leq \frac{\varepsilon n}{t}$. Let $V'_{i,j} = V_{i,j} \setminus U_{i,j}$. 
Similarly, for each $i \in [\ell]$, let $U_{i,D'} = \{u \in Y_i : d_{G_0}(u; X_i) < (d - \varepsilon)|X_i|\}$ and let $Y'_i = Y_i \setminus U_{i,D'}$. Then $|U_{i,D'}| \leq \varepsilon |Y_i| \leq \frac{\varepsilon n}{t}$. Finally, for each $i \in [\ell]$, let $U_{i,D} = \{u \in X_i : d_{G_0}(u; Y_i) < (d - \varepsilon)|Y_i|\}$ and let $X'_i = X_i \setminus U_{i,D}$. Then $|U_{i,D}| \leq \varepsilon |X_i| \leq \frac{\varepsilon n}{t}$.

We collect all the bad vertices, say $Z = V_0 \cup \bigcup_{i, j} U_{i, j} \cup \bigcup_{i} (U_{i, D} \cup U_{i, D'})$. Then $|Z| \leq \ve n + t \frac{\ve n}{t} \leq 2\ve n$. We obtain a graph $G_1 = G_0 - Z$. Then $v(G_1) \geq (1 - 2\ve)n$. We get a set $P'_i = P_i \setminus Z$ for each $i \in [\ell]$.
Then for each $i \in [\ell]$, for every $v \in P'_i$, the degree $d_{G_0}(v; X'_i) \geq (d - \ve)|X_i| - \ve|X_i| \geq (d - 2\ve)|X_i| \geq d_1 |X_i|$. For the same reason, for each $i \in [\ell]$, for every $u \in X'_i$, the degree $d_{G_0}(u; Y'_i) \geq d_1 |Y_i|$.

Let $u \in X'_i$ and $v \in P'_i$, where $i \in [\ell]$. We write $A_u = N_{G_1}(u) \cap Y'_i$ and $B_v = N_{G_1}(v) \cap X'_i$. We observe that the pair $(X_i, Y_i)$ is an $\ve$-regular pair, and the sizes $|A_u| \geq d_1 |Y_i| > \ve |Y_i|$ and $|B_v| \geq d_1 |X_i| > \ve |X_i|$. By the definition of $\ve$-regular pair, the number of edges $e_{G_0}(A_u, B_v) \geq (d - \ve)|A_u||B_v| \geq d_1^3 |X_i||Y_i| \geq \frac{d_1^4}{T^2} n^2$. By \Cref{lem:Kab-supersaturation}, there are at least $d_2(v(T^1_H)!) n^{v(T^1_H)}$ copies of $T^1_H$ in $G_1[A_u, B_v]$. This means there are at least $d_2 n^{v(T^1_H)}$ distinct type-1 $\mathrm{Sub}(H)$-absorbers for $\{u, v\}$. Let $F$ be a $d_2$-absorbing graph of $G_1$. Then the following holds:

\begin{equation}\label{eq:F1}
\text{For each $i \in [\ell]$, every pair $(u, v) \in X'_i \times P'_i$, $uv$ is an edge of $F$.}
\end{equation}

By \Cref{lem:type1-absorber}, there is a set $A \subseteq V(G_1)$ with size at most $\alpha n$ such that for all subsets $U$ of $V(G_1) \setminus A$ with size at most $\beta n$ and $F[U]$ has a perfect matching, then $A$ is a $\mathrm{Sub}(H)$-absorber for $U$. Let us fix such a $\mathrm{Sub}(H)$-absorber $A$. For each $i \in [\ell]$, let $X''_i = X'_i \setminus A$ and $Y''_i = Y'_i \setminus A$. Then we have $|X''_i| \geq |X'_i| - \alpha n \geq (1 - 2\ve) |X_i|$ and, for the same reason, $|Y''_i| \geq (1 - 2\ve) |Y_i|$ for every $i \in [\ell]$.

We now claim that for each $i \in [\ell]$, there is a copy $Q_i$ of a complete balanced bipartite graph $K_{(2h'+1)C, (2h'+1)C}$. We note that the sizes $|X''_i| \geq \ve |X_i|$ and $|Y''_i| \geq \ve |Y_i|$. Since the pair $(X_i, Y_i)$ is an $\ve$-regular pair, the number of edges in $G_1[X''_i, Y''_i]$ is at least $(d - \ve) |X''_i||Y''_i|$. Thus, by the Erdős-Stone-Simonovits theorem, there is at least one copy of $K_{(2h'+1)C, (2h'+1)C}$ in $G_1[X''_i, Y''_i]$. Thus, for each $i \in [\ell]$, there exists such a bipartite graph $Q_i$. Let $Q \defeq \bigcup_{i \in [\ell]} V(Q_i)$. We note that by \Cref{obs:hat-H-reservoir}, for each $i \in [\ell]$, the graph $Q_i$ has a perfect $\hat{H}$-tiling. This implies $G[Q]$ has a perfect $H$-subdivision tiling.

Note that $|Q| = 2(2h'+1)C\ell \leq \ve n$. Let $G_2 = G_0 - (A \cup Q)$. Then we have $v(G_2) \geq (1 - \alpha - \ve)n$ and $\delta(G_2) \geq \left(1 - \frac{1}{\xi(H)} + \frac{\gamma}{2} - (\alpha + \ve)\right)n \geq \left(1 - \frac{1}{\xi(H)} + \frac{\gamma}{4}\right)n$. By \Cref{lem:cover-almost}, there is a set $W \subseteq V(G_2)$ such that $G[W]$ has a perfect $H$-subdivision tiling, the set $Z$ is contained in $W$, and $|V(G_2) \setminus W| \leq C$. Let $J = V(G_2) \setminus W$ with size at most $C$.

We observe that $V(G) = A \cup Q \cup W \cup J$, and $J \subseteq \left(\bigcup_{i \in [\ell]} P'_i \cup \bigcup_{i \in [\ell]} X'_i\right) \setminus (A \cup Q).$ We recall that $A$ is the $\mathrm{Sub}(H)$-absorber for a small vertex subset which has a perfect matching in $F.$
For each $i \in [\ell]$, let $J_i^1 = J \cap X'_i$ and $J_i^2 = J \cap P'_i.$ We denote by $J_i$ the union of $J_i^1$ and $J_i^2.$ Let $c_i$ be an integer for each $i \in [\ell]$ such that $c_i = |J_i^1| - |J_i^2|.$ We define $\sigma(c_i) = 1$ if $c_i > 0$; otherwise, $\sigma(c_i) = 0.$

For each $i \in [\ell]$, if $c_i = 0$, then let $S^1_i$ and $S^2_i$ be empty sets. Otherwise, let them be sets $S^1_i \subseteq (V(Q_i) \cap X'_i)$ with size $(h' + 1 - \sigma(c_i))c_i$ and $S^2_i \subseteq (V(Q_i) \cap Y'_i)$ with size $(h' + \sigma(c_i))c_i.$ By \Cref{obs:hat-H-reservoir}, the induced graph $G[Q_i \setminus (S^1_i \cup S^2_i)]$ has a perfect $H$-subdivision tiling for each $i \in [\ell].$
For each $i \in [\ell]$, since $|S^1_i| - |S^2_i| = -c_i$, the equality $|J^1_i \cup S^1_i| = |J^2_i \cup S^2_i|$ holds. We observe that $(J^1_i \cup S^1_i) \subseteq X'_i$ and $(J^2_i \cup S^2_i) \subseteq P'_i.$ Thus, by \eqref{eq:F1}, the induced graph $F[J_i \cup S^1_i \cup S^2_i]$ has a perfect matching and we have $|J_i \cup S^1_i \cup S^2_i| = |J_i| + |c_i|(2h' + 1).$
Let $S = \bigcup_{i \in [\ell]} S^1_i \cup S^2_i.$ Then, $Q \setminus S = \bigcup_{i \in [\ell]} (V(Q_i) \setminus (S^1_i \cup S^2_i))$, so $G[Q \setminus S]$ has a perfect $H$-subdivision tiling.

Since $J \cup S = \bigcup_{i \in [\ell]} (J_i \cup S^1_i \cup S^2_i)$, the induced graph $F[J \cup S]$ has a perfect matching. Moreover, $|J \cup S| = \sum_{i \in [\ell]} (J_i \cup S^1_i \cup S^2_i) \leq |J| + (2h' + 1) \sum_{i \in [\ell]} |c_i| \leq (2h' + 2)C \leq \beta n$ holds. Thus, the set $A$ is a $\mathrm{Sub}(H)$-absorber for $J \cup S.$ This means the induced graph $G[A \cup J \cup S]$ has a perfect $H$-subdivision tiling. Together with $G[W]$, $G[Q \setminus S]$, and $G[A \cup J \cup S]$, the graph $G$ has a perfect $H$-subdivision tiling. This completes the proof.
\end{proof}

\subsection{Proof of \Cref{thm:hcf-2}}\label{subsec:proof-main2}
We now prove \Cref{lem:lemma-for-main2} to complete the proof of \Cref{thm:hcf-2}. Our purpose is to find a perfect $H$-subdivision tiling in a certain graph, where $\hcf_{\xi}(H) = 2.$

In this case, we cannot apply the same idea as the proof of \Cref{thm:main}. To prove \Cref{thm:hcf-2} with type-$1$ $\mathrm{Sub}(H)$-absorbers, we need to find a perfect matching in an absorbing graph, but since $\hcf_{\xi}(H) = 2$, parity issues make it difficult to obtain a perfect matching. If the host graph $G$ is close to bipartite, we can handle these parity difficulties, but if $G$ is far from bipartite, then we need to use different types of $\mathrm{Sub}(H)$-absorbers. Thus, we divide the cases depending on whether $G$ has many triangles or many $C_5$s. If it has many triangles, we will use type-$2$ $\mathrm{Sub}(H)$-absorbers, and if it has many $C_5$s, then we will use type-$3$ $\mathrm{Sub}(H)$-absorbers.

Below is the proof of \Cref{lem:lemma-for-main2}, which yields \Cref{thm:hcf-2}.

\begin{proof}[Proof of \Cref{lem:lemma-for-main2}]
Let $H$ be a graph on $h$ vertices with $\hcf_{\xi}(H) = 2$. Let $\hat{H}$ be the bipartite graph as in \Cref{def:hatH} with a bipartition $(\hat{A}, \hat{B})$ with sizes $h'$ and $h'+2$ for a positive integer $h'$. We note that $h'$ is bounded above by a constant that only depends on $h$. We fix positive constants as follows:
$$0 < \frac{1}{n}\ll \beta \ll \alpha \ll d_1 \ll \rho \ll \frac{1}{T} \ll \frac{1}{t_0} \ll \ve \ll d \ll \frac{1}{h'}, \frac{1}{C} \ll \gamma, \frac{1}{h} \leq 1.$$
Let $n$ be an even number and let $G$ be an $n$-vertex graph with a minimum degree of at least $\left(1 - \frac{1}{\xi^*(H)} + \gamma \right)n$. We apply the regularity lemma on $G$ with parameters $\ve$, $d$, and $t_0$. Let $V_0, V_1, \dots, V_t$ be an $\ve$-regular partition of $G$ and let $G_0$ be a spanning subgraph of $G$ which satisfies the following:

\begin{itemize}
    \item[(R1)] $t_0 \leq t \leq T_0,$
    \item[(R2)] $|V_0| \leq \ve v(G),$
    \item[(R3)] $|V_1| = \cdots = |V_t|,$
    \item[(R4)] $G_0[V_i]$ has no edges for each $i \in [t],$        
    \item[(R5)] $d_{G_0}(v) \geq d_G(v) - (d + \ve)v(G)$ for all $v \in V(G),$
    \item[(R6)] for all $1 \leq i < j \leq t,$ either the pair $(V_i, V_j)$ is an $(\ve, d+)$-regular pair or $G_0[A, B]$ has no edges.
\end{itemize}

By (R5), the minimum degree of $G_0$ is at least $\left(1 - \frac{1}{\xi^*(H)} + \frac{\gamma}{2} \right)n$. Let $G_1 = G_0 - V_0$, then the size $v(G_1) \geq (1-\ve)n$ and $\delta(G_1) \geq \delta(G_0) - \ve n \geq \left(\frac{1}{3} + \frac{\gamma}{4} \right)n$. Let $R$ be the $(\ve, d)$-reduced graph on $\{V_1, \dots, V_t\}$. Then by \Cref{lem:reduced-minimum-degree}, the minimum degree of $R$ is at least $\left(1 - \frac{1}{\xi^*(H)} + \frac{\gamma}{2} \right)t$. We now divide into two cases depending on the structure of $R$.

\begin{Cases}
    \item $R$ contains $C_3$ or $C_5$.

    Let $j \in \{2, 3\}$. We choose $j$ as 2 if $R$ contains $C_3$. Otherwise, we choose $j$ as 3.
    We may assume $V_1, \dots, V_{2j-1}$ forms a $C_{2j-1}$ in $R$, where $V_i V_{i+1} \in E(R)$ for each $i \in [2j-1]$. Below, we write $V_{2j}$ to denote $V_1$. By \Cref{lem:super-regular-delete}, there exist sets $V'_1, \dots, V'_{2j-1}$ satisfying the following for each $i \in [2j-1]$:

    \begin{enumerate}
        \item[$\bullet$] $V'_i \subseteq V_i$,
        \item[$\bullet$] $|V_i \setminus V'_i| \leq 2\ve |V_i|$,
        \item[$\bullet$] for every $v \in V'_i$, the degrees $d_{G_0}(v; V'_{i-1}) \geq (d - 3\ve)|V_{i-1}|$ and  $d_{G_0}(v; V'_{i+1}) \geq (d - 3\ve)|V_{i+1}|$.
    \end{enumerate}

    Let 
    $$B = \{v \in V(G_1): \text{there are at least } d_1 n^{v(T^j_H)} \text{ distinct type-} j \ \mathrm{Sub}(H) \text{-absorbers for } v\}$$ 

    We now claim that $|B| \geq d_1 n$.

    \begin{claim}\label{clm:size-of-B}
        $|B| \geq d_1 n$.
    \end{claim}

    \begin{claimproof}[Proof of \Cref{clm:size-of-B}]
        If $j = 2$, by the definition of an $\ve$-regular pair, the following holds. For each $i \in [3]$, for every $v \in V'_i$, the number of edges between the sets $N_{G_1}(v) \cap V'_{i-1}$ and $N_{G_1}(v) \cap V'_{i+1}$ is at least $(d - \ve)(d - 3\ve )^2|V_{i-1}||V_{i+1}| \geq \frac{(d - \ve)(d-3\ve)^2(1-\ve)^2}{T^2}n^2 \geq \rho n^2$. Then by \Cref{thm:supersaturation}, there are at least $d_1 (v(T^2_H)!) n^{v(T^2_H)}$ copies of $T^2_H$ in $G_1[N_{G_1}(v)]$. This means there are at least $d_1 n^{v(T^2_H)}$ distinct type-2 $\mathrm{Sub}(H)$-absorbers for $v$.
        This implies $\bigcup_{i \in [3]} V'_i \subseteq B$. Since $|V'_1| + |V'_2| + |V'_3| \geq \frac{3(1-\ve)(1-2\ve)}{T}n \geq d_1 n$, in this case, the inequality $|B| \geq d_1 n$ holds.

        Now we consider the case $j = 3$. We now show that for every $u \in V'_3$, there are many copies of $T^3_H$ such that each copy forms a $\hat{T}^3_H$ together with $u$.
        For each vertex $x \in V'_2$ and $y \in V'_4$, let $N_x = N_{G_1}(x) \cap V'_1$ and $N_y = N_{G_1}(y) \cap V'_5$. By the definition of an $\ve$-regular pair, there are at least $(d-\ve)|N_x||N_y| \geq (d-\ve)(d-3\ve)^2|V_1||V_5| \geq \frac{(d - \ve)(d - 3\ve)^2(1-\ve)^2}{T^2}n^2 \geq \rho n^2$ edges in $G_1[N_x, N_y]$. By \Cref{lem:Kab-supersaturation}, there are at least $\frac{d_1 T^2}{(d - 3\ve)^2(1-\ve)^2} (v(T^3_H)!)n^{v(\Tilde{T}^3_H)}$ distinct copies of $\Tilde{T}^3_H$ such that together with $x$ and $y$, each copy forms a $T^3_H$.
        For each $u \in V'_3$, there are at least $(d-3\ve)^2 |V_2||V_4| \geq \frac{(d-3\ve)^2(1-\ve)^2}{T^2}n^2$ distinct pairs $(x, y) \in V'_2 \times V'_4$ such that $ux$, $uy \in E(G_1)$. Thus, there are at least $d_1 (v(T^3_H)!) n^{v(T^3_H)}$ distinct copies of $T^3_H$ such that together with $u$, each copy forms a $\hat{T}^3_H$. This means, for every $u \in V'_3$, there are at least $d_1 n^{v(T^3_H)}$ distinct type-3 $\mathrm{Sub}(H)$-absorbers for $u$. This means $V'_3 \subseteq B$. Since $|V'_3| \geq \frac{(1-\ve)(1 - 2\ve)}{T}n \geq d_1 n$, we have $|B| \geq d_1 n$.
    \end{claimproof}

    Let the two graphs $F_1$ and $F_2$ be the $d_1$-absorbing graph and the $(d_1, B)$-exchanging graph of $G_1$, respectively. We denote by $X$ the set of vertices of $G_1$ that are not elements of $B$. Assume for a vertex $u \in X$, the inequality $e(G_1[N_{G_1}(u)]) > \rho n^2$ holds. If $j = 2$, then by \Cref{thm:supersaturation}, $u$ should be contained in $B$, a contradiction. If $j = 3$, then $G_1$ does not have a triangle, so a contradiction. Thus, for every vertex $u \in X$, the inequality $e(G_1[N_{G_1}(u)]) \leq \rho n^2$ holds. We now claim the following.

    \begin{claim}\label{clm:XB}
        For every $u \in X$ and $v \in B$, the edge $uv$ is contained in $F_1$ or $F_2$.
    \end{claim}

    \begin{claimproof}[Proof of \Cref{clm:XB}]
        Fix a vertex $v \in B$ and let $U = N_{G_1}(v)$. We partition the set $X$ into $X_1$ and $X_2$ as follows. Let $X_1 = \{u \in X: |N_{G_1}(u) \cap U| \leq \frac{\gamma}{10}n\}$ and $X_2 = X \setminus X_1$.

        Let $u \in X$. Assume $u \in X_1$. Let $Z_1 = N_{G_1}(u)$ and let $Y = V(G_1) \setminus (Z_1 \cup U)$. We denote by $U_1 = U \setminus Z_1$. 
        Since $\delta(G_1) \geq \left(\frac{1}{3} + \frac{\gamma}{4}\right)n$ and $|Z_1 \cap U| \leq \frac{\gamma}{10}n$, the inequality $|Y| \leq \left(\frac{1}{3} - \frac{2\gamma}{5} \right)n$ holds. 
        Thus, for every $u' \in Z_1$, the degree $G_1[u'; Z_1 \cup U]$ is at least $\frac{\gamma}{2}n$. Since $u \in X$, we know that $e(G_1[Z_1]) \leq \rho n^2$. 
        These imply that $e(G_1[Z_1, U_1]) \geq \frac{\gamma}{2}n \left(\frac{1}{3} + \frac{\gamma}{4} \right)n - 2\rho n^2 \geq \frac{\gamma}{10}n^2$ holds. 
        Then by \Cref{lem:Kab-supersaturation}, there are at least $d_1(v(T^1_H)!)n^{v(T^1_H)}$ copies of $T^1_H$ in $G_1[Z_1, U_1]$ such that together with $u$ and $v$, each copy of $T^1_H$ forms a $\hat{T}^1_H$. Thus, there are at least $d_1 n^{v(T^1_H)}$ distinct type-1 $\mathrm{Sub}(H)$-absorbers for $\{u, v\}$. This means $uv \in E(F_1)$.

        We now assume $u \in X_2$. Let $Z_2 = N_{G_1}(u)$. The size $|Z_2 \cap U| \geq \frac{\gamma}{10}n$ causes $u \in X_2$. Since $\delta(G_1) \geq \left(\frac{1}{3} + \frac{\gamma}{4} \right)n$, by \Cref{lem:fix-embedding}, there are at least $d_1 (v(S_H)!)n^{v(S_H)}$ copies of $S_H$ in $G_1 - \{u, v\}$ such that together with $u$ and $v$, each copy of $S_H$ forms a $\hat{S}_H$. This means there are at least $d_1 n^{v(S_H)}$ distinct $\mathrm{Sub}(H)$-exchangers for $\{u, v\}$, so $uv \in E(F_2)$.

        Thus, we conclude that for every $u \in X$ and $v \in B$, the edge $uv$ is contained in $F_1$ or $F_2$.
    \end{claimproof}

    \Cref{clm:XB} implies for all $u \in X$, we have $d_{F_1}(u; B) \geq \frac{|B|}{2} \geq \frac{d_1}{2}n$ or $d_{F_2}(u; B) \geq \frac{|B|}{2} \geq \frac{d_1}{2}n$. We are now ready to apply \Cref{lem:type2-absorber}. By \Cref{lem:type2-absorber}, there is a set $A \subseteq V(G_1)$ with size at most $\alpha n$ such that for any set $U \subseteq V(G_1) \setminus A$ with size at most $\beta n$, the set $A$ is a $\mathrm{Sub}(H)$-absorber for $U$. Let us fix such a $\mathrm{Sub}(H)$-absorber $A$.

    Let $G_2 = G_0 - A$. Then $v(G_2) \geq (1 - \alpha)n$ and $\delta(G_2) \geq \delta(G_0) - \alpha n \geq \left(1 - \frac{1}{\xi(H)} + \frac{\gamma}{4}\right)n$. Then by \Cref{lem:cover-almost}, there is a set $W \subseteq V(G_2)$ such that $V_0 \subseteq W$, the induced graph $G_2[W]$ has a perfect $H$-subdivision tiling, and $|V(G_2) \setminus W| \leq C$. Let $J = V(G_2) \setminus W$. We observe that $V(G) = A \cup W \cup J$ and $J \subseteq V(G_1) \setminus A$. Since $|J| \leq C \leq \beta n$, by our choice of $A$, the induced graph $G_1[A \cup J]$ has a perfect $H$-subdivision tiling. Together with $G[W]$, the graph $G$ has a perfect $H$-subdivision tiling. This completes the proof of Case 1.~\\~

    \item $R$ is a $\{C_3, C_5\}$-free graph.

    In this case, we only use the type-1 $\mathrm{Sub}(H)$-absorbers. The proof is similar to the proof of \Cref{thm:main}, but fortunately, it will be simpler.

    By (R4) and (R6), the graph $G_1$ is also a $\{C_3, C_5\}$-free graph. A classical result of Andrásfai~\cite{andrasfal1964graphentheoretische} states that every graph on $n$-vertices with a minimum degree of at least $\frac{2}{2k+1} n + 1$ either contains an odd cycle with a length of at most $2k-1$ or is bipartite. Since $\delta(G_1) \geq \left(\frac{1}{3} + \frac{\gamma}{4} \right)n > \frac{2}{7}n + 1$ and $G_1$ is a $\{C_3, C_5\}$-free graph, we can deduce that $G_1$ is a bipartite graph.

    Let $(X, Y)$ be the bipartition of $G_1$. Then we have $\delta(G_1) \leq |X| \leq n - \delta(G_1)$, implying $2\delta(G_1) - |X| \geq \frac{3\gamma}{4}n$. Let $x \in X$ and $y \in Y$. We write $N_x = N_{G_1}(x)$ and $N_y = N_{G_1}(y)$. Since $G_1$ is bipartite, we have $N_x \subseteq Y$ and $N_y \subseteq X$. Then 
    $$e(G_1[N_x, N_y]) \geq (\delta(G_1) - |X \setminus N_y|)|N_x| \geq (2\delta(G_1) - |X|)\delta(G_1) \geq \frac{\gamma}{4}n^2.$$ 
    By \Cref{lem:Kab-supersaturation}, there are at least $d_1(v(T^1_H)!)n^{v(T^1_H)}$ copies of $T^1_H$ in $G_1[N_x, N_y]$ such that together with $\{x, y\}$, each copy of $T^1_H$ forms a $\hat{T}^1_H$. This means for every $x \in X$ and $y \in Y$, there are at least $d_1n^{v(T^1_H)}$ distinct type-1 $\mathrm{Sub}(H)$-absorbers for $\{x, y\}$.

    Let $F$ be the $d_1$-absorbing graph for $G_1$. Then by the previous argument, all pairs $xy$ with $x \in X$ and $y \in Y$ are edges of $F$. By \Cref{lem:type1-absorber}, there is a set $A \subseteq V(G_1)$ such that $|A|$ is even, $|A|$ is at most $\alpha n$, and it satisfies the following:

    \begin{itemize}[label = {}]
        \item The set $A$ is a $\mathrm{Sub}(H)$-absorber for any set $U \subseteq V(G_1) \setminus A$ satisfying $|U| \leq \beta n$ and $F[U]$ has a perfect matching.
    \end{itemize}
    We now consider a graph $G_2 = G_1 - A$. Then $V(G_2) \geq (1 - \ve - \alpha)n$ and $\delta(G_2) \geq \delta(G_1) - \alpha n \geq \left(\frac{1}{3} + \frac{\gamma}{10} \right)n$. Let $X' = X \setminus A$ and $Y' = Y \setminus A$. Since $e(G_2[X', Y']) \geq \frac{1}{6}n^2$, by the Erdős-Stone-Simonovits theorem, there is a subgraph $Q$ in $G_1$ which is isomorphic to $K_{(2h'+2)C, (2h'+2)C}$. We note that by \Cref{obs:hat-H-reservoir}, the graph $Q$ has a perfect $\hat{H}$-tiling, so it has a perfect $H$-subdivision tiling.

    Let $G_3 = G_0 - (A \cup V(Q))$. Then $v(G_3) \geq (1 - \alpha)n - (2h' + 2)C$ and $\delta(G_3) \geq \left(1 - \frac{1}{\xi(H)} + \frac{\gamma}{10} \right)n$. Then by \Cref{lem:cover-almost}, there is a set $W \subseteq V(G_3)$ such that $V_0 \subseteq W$ and $G[W]$ has a perfect $H$-subdivision tiling, $|V(G_3) \setminus W|$ is at most $C$, and the size of $W$ is even. Let $J = V(G_3) \setminus W$. We observe that $V(G) = A \cup V(Q) \cup W \cup J$. Note that all $n$, $|A|$, and $|W|$ are even, so is $|J|$. Moreover, $J \subseteq V(G_1) \setminus (A \cup V(Q))$ and $|J| \leq C$.

    Let $J_X = J \cap X$ and $J_Y = J \cap Y$. We may assume $|J_X| - |J_Y| = 2c$ for some non-negative integer $c \leq \frac{C}{2}$. Let us pick two sets $Q_1 \subseteq V(Q) \cap X$ and $Q_2 \subseteq V(Q) \cap Y$ such that $|Q_1| = ch'$ and $|Q_2| = c(h' + 2)$. Then $(J_X \cup Q_1) \subseteq X$ and $(J_Y \cup Q_2) \subseteq Y$. Since $|Q_1| - |Q_2| = -2c$, the equality $|J_X \cup Q_1| = |J_Y \cup Q_2|$ holds. We showed that $F$ contains a complete bipartite graph on the bipartition $(X, Y)$ as a subgraph, and the set $J \cup Q_1 \cup Q_2$ has a perfect matching in $F$. Moreover, we have $|J \cup Q_1 \cup Q_2| \leq (2h' + 2)C \leq \beta n$. Thus, by our choice of $A$, the induced graph $G[A \cup J \cup Q_1 \cup Q_2]$ has a perfect $H$-subdivision tiling.
    Let $Q' = Q - (Q_1 \cup Q_2)$. By \Cref{obs:hat-H-reservoir}, the graph $Q'$ has a perfect $\hat{H}$-tiling, so it has a perfect $H$-subdivision tiling. Together with $G[W]$, $G[A \cup J \cup Q_1 \cup Q_2]$, and $G[V(Q')]$, the graph $G$ has a perfect $H$-subdivision tiling. This completes the proof of \Cref{lem:lemma-for-main2}.

\end{Cases}
\end{proof}

\section{Concluding remarks and open problems}
In this article, we determined an asymptotically tight minimum degree threshold that ensures the existence of a perfect $H$-subdivision tiling for every graph $H$. In many cases, a minimum degree threshold for perfect $H$-subdivision tilings is much smaller than for perfect $H$-tilings since perfect $H$-subdivision tilings are allowed to use not only $H$ but also subdivisions of $H$. To prove that the weaker minimum degree suffices, we developed new approaches using the absorption method combined with the regularity lemma and domination numbers.

Both \Cref{thm:main,thm:hcf-2} are tight up to $o(n)$ terms. We conjecture that our results hold with the minimum degree condition sharp up to additive constants depending only on $H$.

\begin{conjecture}\label{conj:uptoconstant-hcf1}
    Let $H$ be a graph that is not a disjoint union of isolated vertices. Then there exists a constant $C_H$ depending only on $H$ such that the following holds for all $n > 0.$
    
    If $\hcf_{\xi}(H) \neq 2$,
    $$ \delta_{\mathrm{sub}}(n, H) \leq \left(1 - \frac{1}{\xi^*(H)} \right)n + C_H.$$
    
    Otherwise,
    \begin{align*}
        \delta_{\mathrm{sub}}(n, H) &\leq \frac{1}{2}n + C_H &&\text{if } n \text{ is odd,}\\ 
        \delta_{\mathrm{sub}}(n, H) &\leq \left(1 - \frac{1}{\xi^*(H)}\right)n + C_H &&\text{if } n \text{ is even.} 
    \end{align*}
\end{conjecture}

Our main results, \Cref{thm:main,thm:hcf-2}, imply that perfect $H$-subdivision tilings and perfect $H$-tilings behave differently in many cases. In particular, the numbers $\delta_{\mathrm{sub}}(n, K_r)$ behave irregularly for small values of $r$. In contrast to perfect $K_r$-tilings, the number $\lim_{n\to \infty} \frac{\delta_{\mathrm{sub}}(n, K_r)}{n}$ does not strictly increase as $r$ increases.
For another example, consider a tree $T$. According to \Cref{thm:kuhn-osthus}, we have $\delta(n, T) = \frac{1}{2}n + O(1)$. However, for $\delta_{\mathrm{sub}}(n, T)$, \Cref{thm:main} implies that there is a positive constant $c_T$ such that $\delta_{\mathrm{sub}}(n, T) \leq \left(\frac{1}{2} - c_T + o(1) \right)n$ since $\hcf_{\xi}(T) = 1$.

Moreover, as we saw before, monotonicity does not hold for subdivision tilings. For $H_2 \subseteq H_1$ with $H_2$ being a spanning subgraph, $\delta(n, H_2) \leq \delta(n, H_1)$ is trivial, but $\delta_{\mathrm{sub}}(n, H_2) \leq \delta_{\mathrm{sub}}(n, H_1)$ is not true as $\delta_{\mathrm{sub}}(n, K_4) = \frac{2}{5}n + o(n) < \delta_{\mathrm{sub}}(n, C_4) = \frac{1}{2}n + o(n)$.

A perfect $H$-subdivision tiling is a special case of tiling a host graph by several kinds of graphs. Let $\calF$ be a collection of graphs and let $\hcf(\calF)$ be the highest common factor of the sizes of graphs in $\calF$. For each positive integer $n$ divisible by $\hcf(\calF)$, we define $\delta(n, \calF)$ to be the smallest integer $k$ such that any $n$-vertex graph $G$ with a minimum degree of at least $k$ has a perfect tiling that uses only graphs in $\calF$.
In this notation, we estimated an asymptotically tight value for $\delta(n, \mathrm{Sub}(H))$ for every graph $H$. Thus, we suggest the following problem.

\begin{problem}\label{prob:prob}
    Let $\calF$ be a collection of graphs and $n$ be a positive integer that is divisible by $\hcf(\calF)$. Determine the number $\delta(n, \calF).$
\end{problem}

\subsection*{Acknowledgement}
The author would like to thank his advisor, Jaehoon Kim, for his very helpful comments and advice. He would also like to extend his thanks to anonymous reviewers for very carefully reading this article and providing helpful comments.


\begin{thebibliography}{99}

\bibitem{alon1990transversal}
Noga Alon.
\newblock Transversal numbers of uniform hypergraphs.
\newblock {\em Graphs and Combinatorics}, 6(1):1--4, 1990.

\bibitem{alon1996h}
Noga Alon and Raphael Yuster.
\newblock $H$-factors in dense graphs.
\newblock {\em Journal of Combinatorial Theory, Series B}, 66(2):269--282,
  1996.

\bibitem{andrasfal1964graphentheoretische}
B\'{e}la Andr{\'a}sfai.
\newblock Graphentheoretische extremalprobleme.
\newblock {\em Acta Mathematica Hungarica}, 15(3-4):413--438, 1964.

\bibitem{arnautov1974estimation}
Vladimir~I. Arnautov.
\newblock Estimation of the exterior stability number of a graph by means of
  the minimal degree of the vertices.
\newblock {\em Prikl. Mat. i Programmirovanie}, 11(3-8):126, 1974.

\bibitem{balogh2022tilings}
J{\'o}zsef Balogh, Lina Li, and Andrew Treglown.
\newblock Tilings in vertex ordered graphs.
\newblock {\em Journal of Combinatorial Theory, Series B}, 155:171--201, 2022.

\bibitem{bottcher2009proof}
Julia B{\"o}ttcher, Mathias Schacht, and Anusch Taraz.
\newblock Proof of the bandwidth conjecture of Bollob{\'a}s and Koml{\'o}s.
\newblock {\em Mathematische Annalen}, 343(1):175--205, 2009.

\bibitem{erdHos1965limit}
Paul Erd{\H{o}}s and Mikl{\'o}s Simonovits.
\newblock A limit theorem in graph theory.
\newblock In {\em Studia Sci. Math. Hung}. Citeseer, 1965.

\bibitem{erdHos1983supersaturated}
Paul Erd{\H{o}}s and Mikl{\'o}s Simonovits.
\newblock Supersaturated graphs and hypergraphs.
\newblock {\em Combinatorica}, 3(2):181--192, 1983.

\bibitem{erdos1946structure}
Paul Erd{\"o}s and Arthur~H. Stone.
\newblock On the structure of linear graphs.
\newblock {\em Bulletin of the American Mathematical Society},
  52(12):1087--1091, 1946.

\bibitem{freschi2022dirac}
Andrea Freschi and Andrew Treglown.
\newblock Dirac-type results for tilings and coverings in ordered graphs.
\newblock In {\em Forum of Mathematics, Sigma}, volume~10, page e104. Cambridge
  University Press, 2022.

\bibitem{hajnal1970proof}
Andr{\'a}s Hajnal and Endre Szemer{\'e}di.
\newblock Proof of a conjecture of P.Erd\H{o}s.
\newblock {\em Combinatorial theory and its applications}, 2:601--623, 1970.

\bibitem{han2021tilings}
Jie Han, Patrick Morris, and Andrew Treglown.
\newblock Tilings in randomly perturbed graphs: Bridging the gap between
  hajnal-szemer{\'e}di and Johansson-Kahn-Vu.
\newblock {\em Random Structures \& Algorithms}, 58(3):480--516, 2021.

\bibitem{hurley2022sufficient}
Eoin Hurley, Felix Joos, and Richard Lang.
\newblock Sufficient conditions for perfect mixed tilings.
\newblock {\em Journal of Combinatorial Theory, Series B}, 170:128--188 2025.

\bibitem{hyde2019degree}
Joseph Hyde, Hong Liu, and Andrew Treglown.
\newblock A degree sequence Koml{\'o}s theorem.
\newblock {\em SIAM Journal on Discrete Mathematics}, 33(4):2041--2061, 2019.

\bibitem{katona1964graph}
GOH Katona, T~Nemetz, and M~Simonovits.
\newblock On a graph-problem of Tur{\'a}n in the theory of graphs.
\newblock {\em Matematikai Lapok}, 15:228--238, 1964.

\bibitem{kim2019blow}
Jaehoon Kim, Daniela K{\"u}hn, Deryk Osthus, and Mykhaylo Tyomkyn.
\newblock A blow-up lemma for approximate decompositions.
\newblock {\em Transactions of the American Mathematical Society},
  371(7):4655--4742, 2019.

\bibitem{komlos2000tiling}
J{\'a}nos Koml{\'o}s.
\newblock Tiling Tur{\'a}n theorems.
\newblock {\em Combinatorica}, 20(2):203--218, 2000.

\bibitem{komlos2001proof}
J{\'a}nos Koml{\'o}s, G{\'a}bor N. S{\'a}rk{\"o}zy, and Endre Szemer{\'e}di.
\newblock Proof of the Alon--Yuster conjecture.
\newblock {\em Discrete Mathematics}, 235(1-3):255--269, 2001.

\bibitem{komlos1997blow}
J{\'a}nos Koml{\'o}s, G{\'a}bor~N. S{\'a}rk{\"o}zy, and Endre Szemer{\'e}di.
\newblock Blow-up lemma.
\newblock {\em Combinatorica}, 17(1):109--123, 1997.

\bibitem{komlos2002regularity}
J{\'a}nos Koml{\'o}s, Ali Shokoufandeh, Mikl{\'o}s Simonovits, and Endre
  Szemer{\'e}di.
\newblock The regularity lemma and its applications in graph theory.
\newblock {\em Summer school on theoretical aspects of computer science}, pages
  84--112, 2002.

\bibitem{komlos1996szemeredi}
J{\'a}nos Koml{\'o}s and Mikl{\'o}s Simonovits.
\newblock Szemer{\'e}di's regularity lemma and its applications in graph
  theory.
\newblock 1996.

\bibitem{kuhn2009embedding}
Daniela K{\"u}hn and Deryk Osthus.
\newblock Embedding large subgraphs into dense graphs.
\newblock {\em Surveys in combinatorics 2009}, 137–167. London Math. Soc. Lecture Note Ser., 365. Cambridge University Press, Cambridge, 2009

\bibitem{kuhn2009minimum}
Daniela K{\"u}hn and Deryk Osthus.
\newblock The minimum degree threshold for perfect graph packings.
\newblock {\em Combinatorica}, 29(1):65--107, 2009.

\bibitem{kuhn2005large}
Daniela K{\"u}hn, Deryk Osthus, and Anusch Taraz.
\newblock Large planar subgraphs in dense graphs.
\newblock {\em Journal of Combinatorial Theory, Series B}, 95(2):263--282,
  2005.

\bibitem{lo2015f}
Allan Lo and Klas Markstr{\"o}m.
\newblock $F$-factors in hypergraphs via absorption.
\newblock {\em Graphs and Combinatorics}, 31(3):679--712, 2015.

\bibitem{payan1975nombre}
Charles Payan.
\newblock Sur le nombre d'absorption d'un graphe simple.
\newblock {\em Cahiers Centre \'Etudes Rech}, Op\'er. 17, no.~2-4, 307--317, 1975.

\bibitem{rodl2006dirac}
Vojt{\v{e}}ch R{\"o}dl, Andrzej Ruci{\'n}ski, and Endre Szemer{\'e}di.
\newblock A Dirac-type theorem for 3-uniform hypergraphs.
\newblock {\em Combinatorics, Probability and Computing}, 15(1-2):229--251,
  2006.

\bibitem{rodl2010regularity}
Vojt{\v{e}}ch R{\"o}dl and Mathias Schacht.
\newblock Regularity lemmas for graphs.
\newblock In {\em Fete of combinatorics and computer science}, pages 287--325.
  Springer, 2010.

\bibitem{shokoufandeh2003proof}
Ali Shokoufandeh and Yi~Zhao.
\newblock Proof of a tiling conjecture of Koml{\'o}s.
\newblock {\em Random Structures \& Algorithms}, 23(2):180--205, 2003.

\bibitem{szemeredi1975regular}
Endre Szemer{\'e}di.
\newblock Regular partitions of graphs.
\newblock Technical report, Stanford Univ Calif Dept of Computer Science, 1975.

\end{thebibliography}
\end{document}